\documentclass{amsart}
\usepackage{amsfonts,amscd, amssymb}

\newtheorem{theorem}{Theorem}[section]
\newtheorem{lemma}[theorem]{Lemma}
\newtheorem{corollary}[theorem]{Corollary}
\newtheorem{proposition}[theorem]{Proposition}
\theoremstyle{remark}
\newtheorem{remark}[theorem]{Remark}
\theoremstyle{definition}
\newtheorem{definition}[theorem]{Definition}

\numberwithin{equation}{section}
\makeatother

\newcommand{\norm}[1]{\Vert#1\Vert}

\DeclareMathOperator{\Cdb}{{\mathbb C}}
\DeclareMathOperator{\Rdb}{{\mathbb R}}

\DeclareMathOperator{\Ndb}{{\mathbb N}}

\begin{document}

\title[Jordan operator algebras]{Jordan operator algebras: basic theory}

\author{David P. Blecher}
\address{Department of Mathematics, University of Houston, Houston, TX
77204-3008, USA}
\email[David P. Blecher]{dblecher@math.uh.edu}

\author{Zhenhua Wang}
\address{Department of Mathematics, University of Houston, Houston, TX
77204-3008, USA}
\email[Zhenhua Wang]{zhenwang@math.uh.edu}
\thanks{We acknowledge partial support by the National Science Foundation.}

\begin{abstract}
 Jordan operator algebras are norm-closed spaces of operators  on a Hilbert space
which are closed under the Jordan product.
The discovery of the present paper is that there exists a huge and tractable theory of possibly nonselfadjoint Jordan operator   algebras; they are far more similar to associative
operator algebras than was  suspected.   We initiate the theory of such
algebras.
\end{abstract}

\maketitle                   

\section{Introduction} 
An (associative) {\em operator algebra} is a closed associative subalgebra of $B(H)$, for a
complex Hilbert space $H$.    By a {\em  Jordan operator algebra} we
 mean a  norm-closed  {\em  Jordan subalgebra} $A$ of $B(H)$,  
namely a norm-closed   subspace closed under the 
`Jordan product' $a \circ b = \frac{1}{2}(ab+ba)$. Or equivalently,
 with $a^2 \in A$ for all $a \in A$ (that is, $A$ is
closed under squares; the equivalence
 uses the identity $a \circ b = \frac{1}{2} ((a+b)^2 -a^2 -b^2)$).   
Selfadjoint Jordan operator algebras arose in the work of Jordan, von Neumann, and Wigner on the axiomatic
foundations of quantum mechanics.  One expects the `observables' in a quantum 
system to constitute a (real) Jordan algebra, and if one also wants a good functional
calculus and spectral theory one is led to such 
selfadjoint Jordan algebras (known as JC*-algebras).  Nowadays however the interest in Jordan algebras
and related objects is almost exclusively from pure mathematicians (see e.g.\ \cite{Mc}).   
Despite this interest, we are aware of only one paper in the literature 
that discusses Jordan operator algebras
in our sense above, namely the excellent work of Arazy and Solel \cite{AS}.   

The present paper is a step in the direction
of extending the selfadjoint Jordan theory to nonselfadjoint Jordan operator algebras.
 Our main discovery is that there exists an enormous and tractable theory of possibly nonselfadjoint Jordan operator   algebras; they are far more similar to associative 
operator algebras than was  suspected.   Since much of this
parallels the huge existing theory of associative operator algebras there is 
quite a lot to do, and 
we map out here some foundational and main parts of this endeavor.   We also show here
 and in the sequel \cite{BNj} 
that, remarkably, most of the recent theory from papers
of Blecher, Read, Neal, and others (see e.g.\ \cite{BHN, BRI, BRII, BRord, BNII,Bnpi})
generalizes to Jordan operator algebras.
   We are able to include a rather large number of results 
in a relatively short manuscript since many proofs are similar to their
operator algebra counterparts, and thus we often need only discuss the 
new points that arise (which are usually similar to the new points in some previous proofs).
Several of the more interesting questions and challenging aspects of the theory remain to be explored, some examples of
 these are listed at the end of our paper.  However 
one of the biggest open questions which we had at the time of submission 
has now been solved and will be presented together with 
its very many consequences in the sequel with Neal \cite{BNj}.  Many complementary facts and additional theory
will be forthcoming in a work in progress \cite{ZWdraft} by the second author,
and in his PhD thesis.  
In another direction, in \cite{BNjproj} we study
contractive projections on Jordan operator algebras, finding 
Jordan variants of many of the results in \cite{BNpac}, and some 
improvements on a couple of results of that paper, etc.   Indeed the latter project provided the impetus for the present investigation.

\subsection{Structure of our paper}  \label{stru} 
In 
the remaining part of this Section 1 we give some background
and notation,  for example some facts about the related class
of JC*-algebras, second duals of Jordan operator algebras, etc.  Section 2 is largely concerned with
Jordan variants of `classical' or older facts 
from the theory of (associative)
operator algebras.
  It would be useful if the reader was 
familiar with some of the latter theory, as may be found e.g.\ in
\cite{BLM}, where the reader will also find the basic notation and facts from operator space theory.   For example we begin Section 2 with a preliminary
abstract (operator space) characterization of Jordan operator algebras (another is given in the sequel \cite{BNj}).  
Although there are known characterizations of (associative)
operator algebras
(see e.g.\ \cite{BRS,BLM}), there is no
characterization of the Jordan algebra variant that we are aware of
in the literature.  We then discuss unitization and real positivity
in Jordan operator algebras, universal (associative) algebras
enveloping a Jordan operator algebra, contractive approximate identities
(or {\em cai} for short), Cohen factorization for Jordan algebras, and Jordan representations.
We recall that a Jordan representation is a {\em Jordan homomorphism} into $B(H)$
for a Hilbert space $H$.  A  Jordan homomorphism $T : A \to B$ between
Jordan algebras
is of course  a linear map satisfying $T(ab+ba) = T(a) T(b) + T(b) T(a)$ for $a, b \in A$, or equivalently,
that $T(a^2) = T(a)^2$ for all $a \in A$ (the equivalence follows by applying $T$ to $(a+b)^2$). One of
our results for example is that if $A$ has a
{\em Jordan  contractive approximate identity}
(or J-cai for short), that is a net
$(e_t)$  of contractions with $e_t \circ a \to a$ for all $a \in A$,
then $A$ has a real positive (that is, accretive)
net that acts as a cai for the ordinary product in
any containing generated $C^*$-algebra.  This is useful in studying  
functionals and states, and multiplier algebras, towards the end of Section 2.

Section 3 has a common theme of hereditary subalgebras, open projections,
ideals and $M$-ideals.      We define a   {\em hereditary subalgebra}  of a Jordan operator algebra $A$, or {\em HSA} of $A$ for short, to be a Jordan subalgebra possessing a J-cai (this was defined above), which satisfies
$aAa \subset D$  for any $a \in D$ (or equivalently, by replacing $a$ by $a \pm c$, such that
if $a,c \in D$
and $b \in A$ then $abc  + cba \in D$). 
(We show in Theorem  \ref{Ididnt} that HSA's are precisely the
the sets of form  $\overline{\{ x A x : x \in C \}}$  for some  convex
set $C$ of real positive elements of $A$.) 
    Hereditary subalgebras of $C^*$-algebras play a huge role in modern 
$C^*$-algebra theory.  
  At the  time our paper was written 
our theory of  hereditary subalgebras and open projections could only get to a certain point,
since we were blocked  by 
not  having a suitable variant 
of Hay's theorem from \cite{Hay} (see also \cite{BHN,Bnpi}).     Subsequently this obstacle was overcome and
will appear in \cite{BNj}; this also has a raft of applications in noncommutative topology, etc.   We also note that  if $A$ is a JC*-algebra then one may use positive elements in the usual 
operator theoretic sense in the definition and theory of hereditary subalgebra, or of open projections, as opposed to 
working with real positive elements.   We leave this case of the theory to the 
interested reader.  
   Section 4 develops further aspects 
of the theory of real positive elements and
real positive maps in the setting of Jordan operator algebras, ending with a Banach-Stone theorem for Jordan operator algebras.  

We now turn to  background and notation.

\subsection{JC*-algebras}  
The selfadjoint variant of a 
Jordan operator algebra (that is,  a closed selfadjoint subspace of a $C^*$-algebra which is closed under squares) is exactly what is known in the literature as a
JC*-algebra. 
We describe some
basic background
and results about JC*-algebras (see also e.g.\ the texts \cite{Mc,Rod,RPa,Up, Up87,SOJ}
for more details and background to some of the objects mentioned here
in passing).  Using the relation (\ref{abc}) below with $b$ replaced by $b^*$ 
one can see that JC*-algebras
are JC*-triples (that is, closed subspaces of $B(H,K)$, for Hilbert spaces $H$ and $K$, closed under the operation $x x^* x$;
 JC*-triples were originally called J*-algebras in \cite{Har2}), 
and hence are JB*-triples, indeed  are JB*-algebras.  We direct the reader
to  e.g.\ \cite{BFT}
and its sequels by Bunce and Timoney for definitions and 
more on the operator space structure of 
JC*-triples, see also e.g.\ 
the related papers of Neal and Russo \cite{NeRu,NR05,NR14}.      
Such Jordan algebras 
 are nonassociative but there are various variants of the associative law, for example the 
JB*-triple identity, which we will not describe here.   
JC*-algebras are  complexifications of its selfadjoint part, which is a JC-algebra in the
sense of Topping \cite{Top}.  (Conversely, Wright has shown (see \cite{M})
that the complexification of any JB-algebra is a JB*-algebra, in the sense that there 
is a JB*-algebra norm on the complexification.)   
JC*-algebras are very close to $C^*$-algebras, for example they have a positive and increasing  (in the usual
senses) Jordan contractive approximate identity;
see \cite[Proposition 3.5.23]{Rod}.    

Any selfadjoint element
$x$ in a JC*-algebra $A$ generates a $C^*$-algebra and hence is a difference
of two positive elements.   Thus $A = {\rm Span}(A_+)$ in this case.
J*-algebras containing the identity operator are just the unital JC*-algebras, as may be seen 
using
the fact that J*-algebras are closed under  the operation $x y^* x$ (see
p.\ 17 in \cite{Har1}).    There is a similar statement for J*-algebras containing a unitary in the sense of \cite{Har2}
(see e.g.\  \cite[Proposition 7.1]{Har2}).

A contractive Jordan morphism $T$  from a  JC*-algebra $A$
into $B(H)$ preserves the adjoint $*$.    Indeed for
 any selfadjoint element
 $x$ in a JC*-algebra, $T_{| C^*(x)}$ is a $*$-homomorphism
 so $T(x)$ is selfadjoint.
 Thus $T$  is easily seen to be a
 J*-morphism on $A$ in the sense of  \cite{Har2}.   It follows from results in that paper that $T$ is
isometric iff it is one-to-one.     

Of course a Jordan ideal of  a Jordan algebra $A$ 
is a subspace $E$ with  $\eta \circ \xi \in E$
for $\eta \in E, \xi \in A$.  The  Jordan ideals of a  JC*-algebra $A$
coincide with the JB*-ideals of $A$ when $A$ is regarded as a JC*-triple \cite{Har2}.
That is, if $J$ is a closed subspace of a JC*-algebra $A$ such that $a \circ J \subset J$,
then $ab^*c + c b^* a \in J$ whenever $a, b, c \in A$ and at least one of these are in $J$.   This
implies that $J = J^*$ (take $a = c$ in a Jordan approximate identity for $A$, recalling
that every  JC*-algebra has a positive Jordan approximate identity \cite[Proposition 3.5.23]{Rod}).

An interesting class of JC*-algebras are the Cartan factors of type IV (spin factors, see e.g.\ p.\ 16--20 of  \cite{Har1}): these are 
 selfadjoint operator spaces $A$ in $B(H)$ such that $x^2 \in \mathbb{C} I_H$ for all $x\in A$.
 They are isomorphic to a Hilbert space, and contain
no projections except $I$.     They may be constructed
by finding a set of selfadjoint unitaries $\{ u_i: i \in S \}$ with $u_i \circ u_j = 0$ if $i \neq j$
and setting $A = {\rm Span} \{ I, u_i : i \in S \}$.

\subsection{General facts about Jordan operator algebras}   \label{gf}
Jordan subalgebras of commutative (associative)
 operator algebras are of course ordinary (commutative associative)
operator algebras on a Hilbert space, and the Jordan product is the 
ordinary product.    In particular if $a$ is an element in a 
Jordan operator algebra $A$ inside a $C^*$-algebra  $B$,
 then the closed Jordan algebra
generated by $a$ in $A$ equals the closed operator algebra
generated by $a$ in $B$.    We write this as oa$(a)$.

Associative operator algebras and JC*-algebras are of course Jordan operator algebras.
So is $\{ a \in A : a = a^T \}$, for any subalgebra $A$ of $M_n$.      More generally, given a homomorphism $\pi$ and an antihomomorphism $\theta$
on an associative operator algebra $A$, $\{ a \in A : \pi(a) = \theta(a) \}$ is a
  Jordan operator algebra. 
As another example we mention the Jordan subalgebra $\{ (x,q(x)) : x \in B(H) \}$ of 
$B(H) \oplus^\infty Q(H)^{\rm op}$.   Here $q : B(H) \to Q(H)$ is the canonical quotient map onto 
the Calkin algebra.    This space is clearly closed under squares.
 This example has appeared in operator space theory, for example it 
is an operator space with a predual but no operator space predual, and hence is not 
representable completely isometrically and weak* homeomorphically as a weak* closed 
space of Hilbert space operators.   More examples will be considered elsewhere, e.g.\ in 
\cite{BNjproj} we show that the ranges of various natural classes 
of contractive projections on operator algebras 
are Jordan operator algebras, and several other examples are given in \cite{BNj}.

If  $A$ is a Jordan operator subalgebra  of $B(H)$, then the {\em diagonal}
$\Delta(A) = A \cap A^*$  is a JC*-algebra.   If $A$ is
unital then as a  a JC*-algebra $\Delta(A)$ 
  is independent of the Hilbert space $H$.   That is, if $T : A \to B(K)$ is an isometric Jordan homomorphism 
for a Hilbert space $K$, then $T$ restricts to an  isometric Jordan $*$-homomorphism from 
$A \cap A^*$ onto $T(A) \cap T(A)^*$.  This follows from a fact in the last section about contractive Jordan morphisms between  JC*-algebras
preserving the adjoint.   This is also true for nonunital Jordan algebras as we shall see later after Corollary
\ref{MeyerJ}. 
An element $q$ in a Jordan operator algebra $A$
 is called  a {\em projection} if $q^2 = q$ and $\| q \| = 1$
(so these are just the orthogonal projections on the 
Hilbert space $A$ acts on, which are in $A$).    Clearly $q \in \Delta(A)$. 
A projection $q$ in a Jordan operator algebra $A$ will be called
{\em central} if $qxq = q \circ x$ for all $x \in A$.
For $x \in A$, using the $2 \times 2$ `matrix picture' of $x$ with respect to $q$ one sees that 
\begin{equation} \label{qox} qxq = q \circ x \; \; \; \; \; \; \textrm{if and only if}  \; \; \; \; \; \; qx = xq = qxq 
\end{equation}  
with the products $qx$ and $xq$ taken in any $C^*$-algebra containing $A$ as a
Jordan subalgebra).     Note that this implies and is equivalent to that $q$ is central in any generated
(associative) operator algebra, or in a generated $C^*$-algebra.   This notion is independent of the particular generated
(associative) operator algebra since it is captured by the intrinsic formula $qxq = q \circ x$ for $x \in A$.

In a Jordan operator algebra
we have the Jordan identity
$$ (x^2 \circ y) \circ x = x^2 \circ (y \circ x).$$ 
For $a, b, c$ in a Jordan operator algebra
we have 
\begin{equation}  \label{abc} abc + cba = 2 [(a \circ b) \circ c + 
a \circ (b \circ c)] - 2 (a \circ c) \circ b .  \end{equation}
     Hence
if we define a Jordan ideal to be a  subspace $J$  of a Jordan algebra $A$ such that $A \circ J \subset J$, 
then $abc + c b a \in J$ whenever $a, b, c \in A$ and at least one of these are in $J$.       Thus
$A/J$ is a Jordan algebra, but we do not believe it is in general a Jordan operator algebra without extra
conditions on $J$.   
 Putting $a = c$ in the identity (\ref{abc}) above gives 
$2 aba = (ab+ba) a + a (ab+ba) - [a^2 b + b a^2] \in A$,
or 
\begin{equation}  \label{aba}
aba = 2 (a \circ b) \circ a -  a^2 \circ b.
\end{equation}

By a $C^*$-{\em cover} of a Jordan operator algebra we mean a pair
$(B,j)$ consisting of a $C^*$-algebra $B$ generated by $j(A)$, for a completely  isometric Jordan homomorphism
 $j : A \to B$.

Let $A$ be a Jordan subalgebra of a $C^*$-algebra $B$.  Then we may equip the second dual $A^{\ast\ast}$ with a Jordan Arens product as follows. Consider $a\in A, \varphi\in A^*$ and $\eta, \nu \in A^{**}$ . Let $a\circ \varphi$
($=\varphi \circ a$)  be the element of $A^*$ defined by
\begin{align*}
         \langle a\circ \varphi, b\rangle =\langle \varphi, \frac{ab+ba}{2}\rangle
\end{align*}
for any $b\in A$. Then let $\eta\circ \varphi  (=\varphi \circ \eta)$ be the element of $A^*$ defined by
\begin{align*}
\langle \eta\circ \varphi, a \rangle =\langle \eta, a \circ \varphi \rangle.
\end{align*} By definition, the Arens Jordan  product on $A^{**}$ is given by
\begin{align*}
\langle \eta\circ \nu, \varphi \rangle= \langle \eta, \nu \circ \varphi \rangle.
\end{align*} 
This is equal to the Jordan  product in $A^{**}$ coming from the associative
Arens  product  in $B^{**}$, and we have $\eta\circ \nu =  \nu \circ \eta$.  
Indeed suppose that $a_s\in A,$ $b_t\in A$ such that $a_s\to \eta$ and $b_t\to \nu$ in 
the weak* topology. For any $\varphi\in A^{\ast}$ let  $\hat{\varphi}\in B^{\ast}$ be a Hahn-Banach extension.
Note that 
\begin{align*}
	2\langle \eta\circ \nu, \varphi \rangle&= 2\langle \eta, \nu \circ \varphi \rangle=\lim_s\langle \, \nu\circ \varphi, a_s\rangle\\
	&=2\lim_s\lim_t \,
\langle b_t\circ \varphi, a_s\rangle=\lim_s\lim_t \, \langle \varphi, a_sb_t+b_ta_s\rangle.   
\end{align*} 
Note that 
\begin{align*}
2\langle \eta\circ_B \nu, \hat{\varphi}\rangle &=\langle (\eta\nu+\nu\eta), \hat{\varphi}\rangle\\
&=\langle \eta\nu,\hat{\varphi}\rangle+\langle \nu\eta, \hat{\varphi}\rangle\\
&=\langle\eta, \nu\hat{\varphi} \rangle+\langle \nu, \eta\hat{\varphi} \rangle \\
&=\lim_s \, \langle \nu\hat{\varphi}, a_s\rangle+\lim_s\langle \eta \hat{\varphi} , 
b_t \rangle\\  
&=\lim_s\lim_t \, \langle \hat{\varphi}a_s, b_t \rangle + \lim_s\lim_t \, \langle \hat{\varphi} b_t
, a_s \rangle\\
&=\lim_s\lim_t \, \langle \hat{\varphi}, a_sb_t+b_ta_s\rangle\\
&=\lim_s\lim_t \, \langle \varphi, a_sb_t+b_ta_s \rangle .
\end{align*}

Thus the Jordan product in $A^{\ast\ast}$ agrees on $A^{\ast\ast}$
 with the  Jordan  product coming from the associative Arens  product  in $B^{**}$.
We also see from this that $\eta\circ \nu =  \nu \circ \eta$.  
 In any case, the bidual of a Jordan operator algebra is a Jordan operator algebra,
and may be viewed as a Jordan subalgebra of the von Neumann algebra $B^{**}$.

JW*-algebras (that is, weak* closed JC*-algebras) are closed under meets and joins of projections (see \cite[Theorem 6.4]{Top} or \cite[Lemma 4.2.8]{HS}; 
one may also see this since meets and joins may be defined in terms of
 limits formed from $q_1  \cdots q_{n-1} q_n q_{n-1} \cdots
q_1$ and $(q_1 + \cdots q_n)^{\frac{1}{n}}$, both of which make sense in any Jordan Banach algebra).  
Since for any Jordan operator algebra $A$ we have that $A^{**}$ is a Jordan operator algebra
with diagonal 
$\Delta(A^{**})$ a JW*-algebra, it follows that $A^{**}$ is also closed under meets and joins of projections.
In particular $p \vee q$ is the weak* limit of  the sequence
$(p+q)^{\frac{1}{n}}$, for projections $p, q \in A^{**}$.  

By the analogous proof for the operator algebra
case (see 2.5.5 in \cite{BLM}), any contractive (resp.\ completely contractive)
Jordan homomorphism from a
Jordan operator algebra $A$ into a
weak* closed Jordan operator algebra $M$ extends
uniquely to a  weak* continuous
contractive (resp.\ completely contractive)
Jordan homomorphism $\tilde{\pi}:
A^{\ast\ast}\to M.$ 
\section{General theory of Jordan operator algebras}

\subsection{A characterization of  unital Jordan operator algebras}

The following  is an operator space characterization of unital
(or approximately unital) Jordan operator algebras (resp.\ JC*-algebras).
It references however a containing 
operator space $B$, which may be taken to be 
a $C^*$-algebra if one wishes.

\begin{theorem} \label{jbrs} Let $A$ be a unital operator
space (resp.\ operator system) with a bilinear map $m : A \times A \to B$ which is
completely contractive in the sense of
Christensen and Sinclair (see e.g.\ the first paragraph of {\rm 1.5.4} in \cite{BLM}).   Here  $B$ is a unital  
operator space containing $A$ as a unital-subspace 
(so $1_B \in A$) completely isometrically.
 Define $a \circ b = \frac{1}{2}(m(a,b) + m(b,a))$,
and suppose that $A$ is closed under this operation.
 Assume also that   $m(1,a) = m(a,1) = a$ for $a \in A$.   Then $A$ is a unital Jordan operator algebra
(resp.\  JC*-algebra) with Jordan product $a \circ b$.
\end{theorem}   \begin{proof}  We will use 
the injective envelope $I(A)$ and its properties
(see e.g.\ Chapter 4 of \cite{BLM}).
By injectivity, the canonical morphism $A \to I(A)$ extends to a unital completely contractive 
$u : B \to I(A)$.   
By injectivity again, i.e.\ the well known extension theorem 
for completely contractive bilinear maps/the injectivity of the 
Haagerup tensor product, and the universal property
of that  tensor product,
we can use $u \circ m$ to induce a linear complete contraction
$\tilde{m} : I(A) \otimes_h I(A) \to I(A)$.   It is known that $I(A)$ is a unital $C^*$-algebra
(see e.g.\ \cite[Corollary 4.2.8 (1)]{BLM}).   By rigidity of the injective envelope,
$\tilde{m}(1,x) = x = \tilde{m}(x,1)$ for all $x \in
I(A)$.   By the nonassociative case of the
BRS theorem (see e.g.\ 4.6.3 in \cite{BLM}), together with the Banach-Stone theorem for operator algebras (see
 e.g.\ 8.3.13 in \cite{BLM}), $\tilde{m}$
must be the canonical product
map.    Hence for $a, b \in A$,  $u (m(a,b))$ is the product taken in $I(A)$.     Since $$u(m(a,b)) + u(m(b,a)) = u(m(a,b) + m(b,a)) = m(a,b) + m(b,a) = 2 a \circ b  \in A,$$
 $A$ is a unital Jordan subalgebra of $I(A)$.    If in addition $A$ is an operator system then
the embedding of $A$ in $I(A)$ is a complete order embedding by e.g.\ 1.3.3 in \cite{BLM}, so that
$A$ is a  JC*-subalgebra of $I(A)$.   \end{proof}

\begin{remark}
 The unwary reader might have expected a characterization in terms of a 
bilinear map $m : A \times A \to A$ on a unital operator space $A$ such that $m$
 is completely contractive  in the sense of
Christensen and Sinclair, and makes $A$ a Jordan algebra in the algebraic sense.   However it is easy to prove 
that under those hypotheses $A$ is completely
isometrically isomorphic to a commutative operator algebra.    To see this
notice that the  nonassociative case of the
BRS theorem mentioned in the proof shows that $A$ is an associative operator 
algebra.   Since $m(a,b) = m(b,a)$ for $a, b \in A$ we see that $A$ is commutative.
Actually this proof only used the `commutativity' equality in the last sentence, and 
complete contractivity of $m$, not the other aspects of being a Jordan product, such as the Jordan axiom.
We are indebted to the referee for pointing this out.	
\end{remark}

There is an `approximately unital' analogue of Theorem \ref{jbrs}.  Define an
{\em approximately unital} operator space $A$ to be a
subspace of an approximately unital operator algebra $B$,
such that $A$ contains a cai $(e_t)$ for $B$.
The hypothesis that $m(e_t,a) \to a$ and $m(a,e_t) \to 
 a$ for $a \in A$ is shown in the later result Lemma \ref{jcai} to be 
a reasonable one: in a Jordan operator algebra
with a cai satisfying $e_t  \circ a \to a$, one can 
find another cai satisfying 
$e_t a \to a$ and $a e_t  \to a$ with products here in any 
$C^*$-algebra or 
approximately
unital operator algebra 
containing $A$ as a closed Jordan subalgebra.

\begin{theorem} \label{jbrsau} Let $A$ be  an approximately 
unital operator
space (resp.\ operator system) containing a cai $(e_t)$ for an operator algebra $B$ 
as above.   Let $m : A \times A \to B$ be a 
completely contractive bilinear map in the sense of
Christensen and Sinclair.
 Define $a \circ b = \frac{1}{2}(m(a,b) + m(b,a))$,
and suppose that $A$ is closed under this operation.
 Assume also that   $m(e_t,a) \to a$ and $m(a,e_t) \to
 a$ for $a \in A$.   Then $A$ is a 
 Jordan operator algebra
(resp.\  JC*-algebra) with Jordan product $a \circ b$, and $e_t  \circ a \to a$ for $a \in A$.
\end{theorem}   \begin{proof}    Let $e = 1_{B^{**}} \in A^{**}$.
We consider the canonical weak* continuous extension
$\tilde{m} : A^{**} \times A^{**} \to B^{**}$.
By standard approximation arguments
$\tilde{m}(e , a) = \tilde{m}(a,e) = a$ for all $a \in A^{**}$,
and $\tilde{m}(a,b) + \tilde{m}(b,a)  \in A^{**}$ for all $a, b\in A^{**}$.
By Theorem \ref{jbrs} we have that 
$A^{**}$ is a unital Jordan operator algebra
(resp.\  JC*-algebra) with Jordan product $\frac{1}{2}
(\tilde{m}(a,b) + \tilde{m}(b,a))$.
Hence $A$ is a Jordan operator algebra 
with Jordan product $a \circ b$.
Clearly $e_t  \circ a \to a$.
\end{proof}

Approximately
unital Jordan  operator algebras will be studied in much greater detail 
later.

\subsection{Meyer's theorem, unitization, and real positive elements} 
\label{MeyerRP}  

The following follows from Meyer's theorem on the  unitization of operator algebras (see e.g.\ 2.1.13 and 2.1.15
in \cite{BLM}).

\begin{proposition}  \label{MeyerJ1}  If $A$ and $B$
are Jordan subalgebras of $B(H)$ and $B(K)$  respectively, with $I_H \notin A$, and if $T : A \to B$ is a contractive (resp.\ isometric)
 Jordan homomorphism, then there is a unital  contractive (resp.\ isometric)
 Jordan homomorphism extending $T$ from $A + \Cdb I_H$ to $B + \Cdb I_K$ (for the isometric case
we also need $I_K \notin B$).  \end{proposition}

  \begin{proof}    It is only necessary to show that if   $a \in A$
is fixed, then
Claim: $\Vert T(a) + \lambda
1_K \Vert \leq \Vert a + \lambda
1_H \Vert$ for  $\lambda \in \Cdb$.    However the restriction of $T$ to oa$(a)$ is
an algebra  homomorphism into oa$(T(a))$, and so the Claim follows from  Meyers result. 
\end{proof}

\begin{corollary}[Uniqueness of unitization for Jordan operator algebras]  \label{MeyerJ}   The unitization 
$A^1$ of a Jordan operator algebra is unique up to 
isometric
 Jordan isomorphism.   
 \end{corollary}

 \begin{proof}   If $A$ is nonunital then this follows from Proposition  \ref{MeyerJ1}.   If $A$ is unital, and 
$A^1$ is a unitization on which $A$ has codimension $1$, then since the identity $e$ of $A$ is easily 
seen to be a central 
projection in $A^1$ we have $\Vert a + \lambda 1 \Vert = \max \{ \Vert a + \lambda e \Vert , |\lambda| \}$.
\end{proof}

Because of Corollary \ref{MeyerJ}, for a Jordan operator algebra $A$ 
we can define unambiguously ${\mathfrak F}_A = \{ a \in A : \Vert 1 - a \Vert \leq 1 \}$.
The diagonal $\Delta(A) = A \cap A^* = 
\Delta(A^1) \cap A$ is a JC*-algebra,  as is easily seen, 
and now it is clear that as a JC*-algebra $\Delta(A)$ is independent of the particular Hilbert space
$A$ is represented on (since this is true for $\Delta(A^1)$ as we said in Subsection \ref{gf}).    That is, if $T : A \to B(K)$ is an isometric Jordan homomorphism 
for a Hilbert space $K$, then $T$ restricts to an  isometric Jordan $*$-homomorphism from 
$A \cap A^*$ onto $T(A) \cap T(A)^*$.  
Every  JC*-algebra is approximately
unital (see \cite[Proposition 3.5.23]{Rod}).

If $A$ is a unital Jordan operator subalgebra of $B(H)$, with $I_H \in A$ then $A_{\rm sa}$ makes sense,
and is independent of $H$, these are the hermitian elements in $A$ (that is, $\Vert \exp(it h) \Vert = 1$ for
real $t$; or equivalently $\varphi(h) \in \Rdb$ for all states $\varphi$ of $A$,
where by `state' we mean a unital contractive functional).   Similarly, ${\mathfrak r}_A$,
the real positive or accretive elements in $A$, may be defined as the 
set of $h \in A$ with  Re $\varphi(h) \geq 0$ for all states $\varphi$ of $A$.
This is equivalent to all the other usual conditions 
characterizing accretive elements
(see e.g.\ \cite[Lemma 2.4 and Proposition 6.6]{Bsan};
 some of these use the fact that the Jordan algebra generated by a single element and $1$ is
an algebra).     

If $A$ is a possibly nonunital 
Jordan operator algebra we define ${\mathfrak r}_A$ to be the
elements with positive real part--we call these the {\em real positive} or {\em accretive}  elements of $A$.  Since the unitization is well defined by Proposition  \ref{MeyerJ1}, so is
${\mathfrak r}_A$.   Alternatively, note that $A^1 + (A^1)^*$, and hence  $A + A^*$, is well defined as a unital selfadjoint
subspace independently (up to unital (positive) isometry) of the particular Hilbert space that $A^1$ is represented
isometrically and nondegenerately, by \cite[Proposition 1.2.8]{Arv}.   That is, 
a unital Jordan isometry $T : A^1 \to B(K)$ extends uniquely to a unital positive Jordan isometry $A^1 + (A^1)^*
\to T(A^1) + T(A^1)^*$.  
  Thus a statement such as
$a + b^* \geq 0$ makes sense whenever $a, b \in A$, and is independent of the particular $H$ on which $A$
is represented as above. This gives another way of seeing that  the set ${\mathfrak r}_A = \{ a \in A : a + a^* \geq 0 \}$ is
independent of the particular Jordan representation of $A$ too.   

We have $x \in {\mathfrak c}_A = \Rdb_+ {\mathfrak F}_A$  iff there is a positive
constant $C$ with $x^* x \leq C(x+x^*)$ (to see this note that $\| 1 - t x \|^2 \leq 1$ iff
$(1 - t x)^* (1 - t x) \leq 1$).    Also,  ${\mathfrak r}_A$
is a closed cone in $A$, hence is Archimidean (that is, $x$ and $x + ny \in {\mathfrak r}_A$ for all
$n \in \Ndb$ implies that $y \in {\mathfrak r}_A$).  On the other hand
${\mathfrak c}_A = \Rdb_+ {\mathfrak F}_A$ is not closed in general, but it is a proper cone
(that is, ${\mathfrak c}_A \cap (-{\mathfrak c}_A) = (0)$).     This follows from the proof of the analogous
operator algebra result in the
introduction to \cite{BRord}, since $1$ is an extreme point of the ball of any unital Jordan algebra $A$ since e.g.\ $A$ is a unital 
subalgebra of a unital $C^*$-algebra $B$ and $1$ is 
 extreme   in 
Ball$(B)$.

If $A$ is a nonunital Jordan subalgebra of a unital $C^*$-algebra $B$  then we 
can identify $A^1$ with $A + \Cdb 1_B$, and it follows that
${\mathfrak F}_A = {\mathfrak F}_B  \cap A$ and 
${\mathfrak r}_A = {\mathfrak r}_B  \cap A$.     
Hence if $A$ is a Jordan subalgebra of a Jordan operator algebra $B$
then ${\mathfrak r}_A = {\mathfrak r}_{B} \cap A$
and ${\mathfrak r}_A = {\mathfrak r}_{B} \cap A$.  

\subsection{Universal algebras of a Jordan operator algebra}  \label{univ}  
There are maximal and minimal associative algebras generated by a Jordan operator algebra $A$.
Indeed consider the direct sum $\rho$ of `all' contractive  
 (resp.\ completely contractive) Jordan representations
$\pi : A \to B(H_\pi)$.    There are standard ways to avoid the set theoretic issues with the `all' here--see
e.g.\ the proof of \cite[Proposition 2.4.2]{BLM}.  Let $C^*_{{\rm max}}(A)$ be the $C^*$-subalgebra of $B(\oplus_\pi \, H_\pi)$
generated by $\rho(A)$. 
For simplicity we describe the 
`contractive' case, that is the Banach space rather than operator space  variant of $C^*_{{\rm max}}(A)$.
If one needs the operator space version, simply replace word `contractive' by `completely contractive' in the
construction below.    
The compression map $B(\oplus_\pi \, H_\pi) \to B(H_\pi)$ is a $*$-homomorphism
when restricted to  $C^*_{{\rm max}}(A)$.   It follows that $C^*_{{\rm max}}(A)$ 
is the `biggest' $C^*$-cover of $A$: indeed
 $C^*_{{\rm max}}(A)$ has the universal
property that for every contractive   Jordan representation
$\pi : A \to B(H_\pi)$, there exists a unique $*$-homomorphism $\theta : C^*_{{\rm max}}(A)
\to B(H_\pi)$ with $\theta \circ \rho = \pi$.    One may also construct $C^*_{{\rm max}}(A)$  using the 
method in II.8.3.1 in \cite{Bla}.

We define oa$_{{\rm max}}(A)$  to be the operator algebra generated by $\rho(A)$ inside 
$C^*_{{\rm max}}(A)$.   Again we focus on the Banach space rather than operator space  variant,
if one needs the operator space version, simply replace word `contractive' by `completely contractive' below.
 This  has the universal
property that for every contractive Jordan representation
$\pi : A \to B(H_\pi)$, there exists a unique contractive   
homomorphism $\theta : {\rm oa}_{{\rm max}}(A)
\to B(H_\pi)$ with $\theta \circ \rho = \pi$.    

It follows that if $A$ is a Jordan subalgebra of an approximately unital operator algebra $C$ (resp.\ of a $C^*$-algebra $B$),  such that $A$ generates $C$ as an operator algebra (resp.\ $B$ as a $C^*$-algebra), 
then  there exists a unique  contractive homomorphism $\theta$ from  oa$_{{\rm max}}(A)$  (resp.\
$C^*_{{\rm max}}(A)$) into $C$  (resp.\ onto $B$) with $\theta(\rho(a)) = a$ for all $a \in A$.
Similarly, if $j : A \to C$ is a contractive 
 Jordan homomorphism
and $C$ is a closed Jordan subalgebra of a $C^*$-algebra $B$, then  there exists a unique  
contractive homomorphism $\theta$ from  oa$_{{\rm max}}(A)$  (resp.\
$C^*_{{\rm max}}(A)$) into $C$  (resp.\ into $B$) with $\theta(\rho(a)) = j(a)$ for all $a \in A$.
 
We now turn to the $C^*$-envelope, or `minimal' $C^*$-cover of $A$.   This is necessarily an
operator space object (we must use completely 
contractive homomorphisms).  
If $A$ is a unital Jordan operator algebra then there is a
unital complete isometry $j : A \to C^*_e(A) \subset I(A)$.
Let $i : A \to B$ be a completely isometric Jordan 
homomorphism into a $C^*$-algebra $B$ generated by $i(A)$.   Then 
$i(1)$ is a projection in $B$, and $i(1) \circ i(a) = i(a)$ for all $a \in A$.
This forces $i(a) i(1) = i(a)$ as in our discussion of central projections
in the introduction.  It follows easily that $i(1) = 1_B$, so that $i$ is
unital.
Let  $\theta : B \to C^*_e(A)$ be the $*$-homomorphism 
coming from the universal property of $C^*_e(A)$ (see e.g.\ \cite[Theorem 4.3.1]{BLM}), 
that is $j = \theta \circ i$.
We see that $j$ is a unital Jordan homomorphism, and so $A$ 
`is' a Jordan subalgebra of $C^*_e(A)$.
From its universal property as usual
 $C^*_e(A)$ is in some sense the `smallest' $C^*$-cover of $A$.  

If $A$ is an approximately  unital Jordan operator algebra we define $C^*_e(A)$  to be the 
$C^*$-algebra $D$ generated by $j(A)$ inside $(C^*_e(A^1),j)$, where $A^1$ is the unitization.   
We will discuss this further after we have studied the unitization in 
the approximately unital case in the next subsection.

Define oa$_e(A)$ to be the operator algebra generated by $j(A)$ in
$C^*_e(A)$.

Finally, there are universal JC*-algebra envelopes of a Jordan operator algebra $A$.
Namely, consider the JC*-subalgebra of $C^*_{\rm max}(A)$ generated by $A$.   
This  clearly has the universal
property that for every contractive  (again there is
a  completely contractive version that is almost identical) Jordan representation
$\pi : A \to B(H)$, there exists a unique contractive   
Jordan $*$-homomorphism $\theta : C^*_{{\rm max}}(A)
\to B(H_\pi)$ with $\theta \circ \rho = \pi$.
If $A$ is also  approximately  unital
we may also consider the JC*-subalgebra of $C^*_e(A)$ generated by $A$.
This will have a universal
property similar to that a few paragraphs up,
or in Proposition \ref{3auce} below, but addressing JC*-algebras $B$ 
generated by a completely isometric Jordan
homomorphic copy of $A$.

\subsection{Contractive approximate identities and consequences}

If $A$ is a Jordan operator subalgebra  of a $C^*$-algebra $B$  
 then we say that  a net $(e_t)$ in Ball$(A)$  is 
a $B$-{\em relative partial cai} for $A$ if
 $e_t a \to a$ and $a e_t \to a$  for all $a \in A$.  Here we are using the usual
product on $B$, 
which may not give an element in $A$, and may depend on $B$. 
Nonetheless the 
existence of such a cai is independent of $B$,
 as we shall see.   
We say that a net $(e_t)$ in Ball$(A)$  is 
a  {\em partial cai} for $A$ if 
for every $C^*$-algebra $B$ containing $A$ as
a Jordan subalgebra,  $e_t a \to a$ and $a e_t \to a$  for all $a \in A$,
using the product 
on $B$.   We say that 
$A$ is {\em approximately unital} if it has a partial cai.  

If $A$ is an operator algebra or Jordan operator algebra then we recall that  a net $(e_t)$ in Ball$(A)$  is 
a {\em Jordan cai} or {\em J-cai} for $A$ if  $e_t a + a e_t \to 2a$  for all $a \in A$.  

\begin{lemma} \label{jcai}  If $A$ is a Jordan operator subalgebra 
of a $C^*$-algebra $B$, then the following are equivalent:
\begin{itemize} \item [(i)] $A$ has a partial cai.
 \item [(ii)]  $A$ has a $B$-relative  partial cai. 
\item [(iii)]  
$A$ has a J-cai.
\item [(iv)]   $A^{**}$ has an identity $p$ of norm 1 with respect to  the Jordan Arens product  on $A^{**}$,
which coincides on $A^{**}$ with the restriction of the usual  product
in $B^{**}$.     Indeed $p$  is the identity of the von Neumann algebra
$C^*_B(A)^{**}$.  
\end{itemize}  If these 
hold then $p$ is an open projection in $B^{**}$ in the sense 
of Akemann {\rm \cite{Ake2}}, and any partial cai $(e_t)$ for $A$  
is a cai for $C^*_B(A)$ (and for ${\rm oa}_B(A)$), and every J-cai for $A$ converges weak* to 
$p$.    
\end{lemma}

\begin{proof}   That (i) $\Rightarrow$ (ii) and  (ii) $\Rightarrow$ (iii) are obvious.

(iii) $\Rightarrow$ (iv) \ If $p$ is a weak* limit point of $(e_t)$,  then in the weak* limit we have 
$p a + a p = 2a$ for all $a \in A$, hence for all $a \in A^{**}$.
Thus  $p$ is an identity for the Jordan product of $A^{**}$.
In particular, $p^2 = p$, and so $p$ is an orthogonal projection in $B^{**}$.    It thenfollows from  (\ref{qox}) that $\eta p = p \eta = \eta$ for all $\eta \in A^{**}$
in the $B^{**}$ product.  So $p$ is an identity in the $B^{**}$ product on $A^{**}$ (which may not map into $A^{**}$ if
$A$ is not an operator algebra).   
By topology it now follows that  every J-cai for $A$ converges weak* to 
$p$.    

(iv) $\Rightarrow$ (i) \ Suppose that $p$ is an orthogonal projection in $A^{**}$.   Then by (\ref{qox}), $\eta p = p \eta = \eta$ 
for all $\eta \in A^{**}$ iff $p$ is an identity in the Jordan product on $A^{**}$.   We may replace $B$
with 
$D  = C^*_{\rm max}(A)$, letting $\rho$ be the usual inclusion of $A$ in
$D$.   We may then  follow a standard route to obtain a cai, see e.g.\ the last part of the proof of \cite[Proposition 2.5.8]{BLM}. That is we begin by choosing a net $(x_t)$ in ${\rm Ball}(A)$ with  $e_t \to p$ weak*.   In $D$ we have $a e_t \to a$ and $e_t a \to a$ weakly.
Thus for any finite set $F = \{ a_1, \cdots, a_n \} \subset A$ the zero vector is in the weak and norm closure in $D^{(2n)}$ of 
$$\{ a_1 u - a_1 , \cdots , a_n u - a_n, u a_1  - a_1 , \cdots , u a_n  - a_n  ) : u \in \Lambda \}.$$
From this one produces, by the  standard method in  e.g.\ the last part of the proof of \cite[Proposition 2.5.8]{BLM},
a $D$-relative partial cai $(e_t)$  for $A$ formed   from
convex combinations.    Suppose that $B$ is any $C^*$-algebra containing $A$ as a Jordan subalgebra
via a completely isometric inclusion $i : A \to B$, such that $B = C^*_B(i(A))$.  Then the existence of a canonical   $*$-homomorphism $\theta : C^*_{{\rm max}}(A)
\to B$ with $\theta \circ \rho = i$, gives $i(e_t) i(a) = \theta(\rho(e_t) \rho(a)) \to \theta(\rho(a)) =
 i(a)$ for $a \in A$,
and similarly $i(a) i(e_t) \to i(a)$.   
So $(e_t)$ is a partial cai for $A$.

If these hold, and if $A$ is a Jordan 
subalgebra of a $C^*$-algebra $B$, then since $e_t = e_t p \to p$ weak*, we have as in the operator algebra case that $p$ is open in $B^{**}$.     Note that $C^*_B(A)$ is a $C^*$-algebra with cai $(e_t)$ by 
\cite[Lemma 2.1.6]{BLM}, and so $p = 1_{C^*_B(A)^{**}}$.  \end{proof} 

\begin{remark} It follows from the last result that any $C^*$-cover $(B,j)$
of an approximately unital Jordan operator algebra
$A$ is a unital $C^*$-algebra if and only if $A$ is
unital.   Indeed 
if $B$ is unital then $1_B = \lim_t \, e_t \in A$.
 Similarly oa$_B(A)$ is unital if and only if $A$ is
unital. 
\end{remark}

 If $A$ is a Jordan operator algebra in a $C^*$-algebra (resp.\ operator algebra) $B$ and
$(e_t)$ is a partial cai for $A$, then it follows
from the above that 
$\{ T \in B : T e_t \to T, e_t T  \to T \}$ is a 
$C^*$-algebra with cai $(e_t)$ containing $A$.

\begin{theorem} \label{kapl}  If $A$ is an approximately unital  Jordan  operator algebra then
$A$ is an $M$-ideal in $A^1$.   Also ${\mathfrak F}_A$ is weak* dense in ${\mathfrak F}_{A^{**}}$
and ${\mathfrak r}_A$ is weak* dense in ${\mathfrak r}_{A^{**}}$. 
Finally, $A$ has a partial cai in $\frac{1}{2} {\mathfrak F}_A$.
\end{theorem}

\begin{proof}   As in the operator algebra case, if $e$ is the identity of $A^{**}$ viewed as a (central) projection in $(A^1)^{**}$, then multiplication by 
$e$ is an $M$-projection
from $(A^1)^{**}$ onto $A^{**}$.  This gives the first assertion 
and the assertions about weak* density are
identical to the proof in \cite[Theorem 5.2]{BOZ}.  One does need to 
know that for $x \in {\mathfrak r}_A$ we have 
$x(1+x)^{-1} \in \frac{1}{2} {\mathfrak F}_A$, but this is easy as we 
can work in the operator algebra oa$(x)$.   

Note that the identity of $A^{**}$ is in  $\frac{1}{2} {\mathfrak F}_{A^{**}}$.
Hence by the just established weak* density,
 it may be approximated by a net in $\frac{1}{2} {\mathfrak F}_A$.  From this as in Lemma \ref{jcai} (but taking $\Lambda = \frac{1}{2} {\mathfrak F}_A$) one may construct a cai in $\frac{1}{2} {\mathfrak F}_A$ in a  standard way  from convex combinations.
Alternatively, one may copy the proof of  Read's theorem in \cite{Bnpi}
to get this. 
\end{proof}

\begin{remark} Indeed as in \cite[Theorem 2.4]{BRI} by taking $n$th roots
one may find in any approximately unital Jordan operator algebra,
a partial cai
which is {\em nearly positive} in the sense described in the introduction
of \cite{BRord}. 
\end{remark}
 
The following is a  `Kaplansky density' result for ${\mathfrak r}_{A^{**}}$:  

\begin{proposition}  \label{Kap} Let $A$ be an approximately unital Jordan operator algebra.  Then the set of contractions
in ${\mathfrak r}_A$ is weak* dense in the set of contractions
in ${\mathfrak r}_{A^{**}}$. 
\end{proposition}
 \begin{proof}    We showed in Theorem 2.6 that ${\mathfrak r}_A$ is weak* dense in ${\mathfrak r}_{A^{**}}$.
The bipolar argument in \cite[Proposition 6.4]{BOZ} (but replacing appeals to results in that paper
by  appeals to the matching results in the present paper) shows that  ${\rm Ball}(A) \cap {\mathfrak r}_A$ is weak* dense in 
${\rm Ball}(A^{**}) \cap {\mathfrak r}_{A^{**}}$.
 \end{proof}

\begin{corollary}
If $A$ is a Jordan operator algebra with a countable Jordan cai  $(f_n)$, then $A$ has a countable partial cai in $\frac{1}{2} {\mathfrak F}_A.$  \end{corollary}
\begin{proof}
By Theorem \ref{kapl}, $A$ has a partial cai	$(e_t)$ in $\frac{1}{2} {\mathfrak F}_A.$ Choose $t_n$ with $\Vert f_ne_{t_n}-f_n\Vert \vee \Vert e_{t_n}f_n-f_n\Vert <2^{-n}$.  It is easy to see that $(e_{t_n})$ is a countable partial cai in $\frac{1}{2} {\mathfrak F}_A.$
\end{proof}

\begin{proposition} \label{ununau}  If $A$ is a nonunital  approximately unital Jordan operator algebra 
then the unitization $A^1$ is well defined 
 up to completely isometric Jordan isomorphism; the matrix norms are
$$\Vert [ a_{ij} + \lambda_{ij} 1 ] \Vert
= \sup \{ \Vert [ a_{ij} \circ c+ \lambda_{ij} c \Vert_{M_n(A)}
: c \in {\rm Ball}(A) \} , \qquad a_{ij} \in A, \lambda_{ij} \in \Cdb.$$
\end{proposition}

\begin{proof}    Suppose that  $\pi : A \to B(H)$ is a completely isometric Jordan
homomorphism.   Let $B$ be the $C^*$-subalgebra
of $B(H)$ generated by $\pi(A),$ let $K= \overline{BH}$, and let $e = P_K$ be 
the projection onto $K$.   It will be clear to $C^*$-algebraists, and  follows  easily
from e.g.\ \cite[Lemma 2.1.9]{BLM}, that 
$\Vert \pi(b) + \lambda I_H \Vert = \max \{ |\lambda| , \Vert \pi(b)_{| K} + \lambda I_K \Vert \}$ for $b \in B, \lambda \in \Cdb$,
and similarly at the matrix level.   
 Thus we may suppose that $H = K$.    Note that $I_K \notin B$, since if it were then
the identity $e$ of $\pi(A)^{**}$ is the identity
of $B^{**}$ by  Lemma \ref{jcai}, which is $I_K \in B$, so that $e \in B \cap \pi(A)^{**} = \pi(A)$, contradicting
that $A$ is nonunital.  
By  Lemma \ref{jcai},
for any  partial cai  $(e_t)$ for $A$, $(\pi(e_t))$  is a partial cai for $\pi(A)$ and hence it
(and also $(\pi(e_t)^*)$) is a cai for $B$
by Lemma \ref{jcai}.    Since $B$ acts nondegenerately on $H$ it is clear that 
 $\pi(e_t) \to I_H$ and $\pi(e_t)^* \to I_H$ strongly on $H$.   
It  is easy to see that
$\Vert [\pi(a_{ij})  + \lambda_{ij} I_H] \Vert$ equals
$$\sup \{ |\sum_{i,j} \langle 
(\pi(a_{ij}) + \lambda_{ij} I_H) \zeta_j, \eta_i \rangle | \}
 = \sup \{ \lim_t \, | \sum_{i,j} \,  
\langle (\pi(a_{ij}) \circ \pi(e_t) + \lambda_{ij}  \pi(e_t))
 \zeta_j, \eta_i \rangle | \},$$ supremum over $\zeta, \eta \in {\rm Ball}(H^{(n)})$.  
This is dominated by 
$$\sup_t  \Vert [\pi(a_{ij} \circ e_t + \lambda_{ij} e_t) ] \Vert   \leq \sup \Vert [ a_{ij} \circ c + \lambda_{ij} c ] \Vert
: c \in {\rm Ball}(A) \} .$$
In turn the latter equals 
$$\sup \Vert [ \pi(a_{ij}) \circ \pi(c) + \lambda_{ij} \pi(c)] \Vert
: c \in {\rm Ball}(A) \}
\leq \Vert [\pi(a_{ij})  + \lambda_{ij} I_H] \Vert .$$  
This proves the assertion.
\end{proof}  

\begin{remark} (1) \  By the proof one may replace $\circ c$ in the last statement with
$\circ e_t$, and take the supremum over all $t$.

\smallskip

(2) \ A similar argument, but replacing a term above by
$\sup \{ \lim_t \, | \sum_{i,j} 
\langle (\pi(a_{ij}) \pi(e_t) + \lambda_{ij}  \pi(e_t))
 \zeta_j, \eta_i \rangle | \}$, shows that
$$\Vert [ a_{ij} + \lambda_{ij} 1 ] \Vert
= \sup \{ \Vert [ a_{ij} c+ \lambda_{ij} c ] \Vert_{M_n(A)}
: c \in {\rm Ball}(A) \} $$
where the product $a_{ij} c$  is with respect to (any fixed) containing
$C^*$-algebra.
One also has similar formulae
with $a c$ replaced by $ca$.
\end{remark}

If $A$ is an approximately  unital Jordan operator algebra we define the $C^*$-envelope $C^*_e(A)$ to be the
$C^*$-algebra $D$ generated by $j(A)$ inside $(C^*_e(A^1),j)$, where $A^1$ is the unitization.

\begin{proposition} \label{3auce}  Let  $A$ be an approximately  unital Jordan
operator algebra, and let  $C^*_e(A)$ and $j$ be as defined above. Then $j_{\vert A}$ is a Jordan homomorphism onto a Jordan subalgebra of $C^*_e(A)$, and $C^*_e(A)$ has the following
universal property: Given any $C^*$-cover $(B,i)$
of $A$, there exists a (necessarily unique and surjective)
$*$-homomorphism $\theta  \colon B \rightarrow C^*_e(A)$ such that
$\theta \circ i = j_{\vert A}$.
\end{proposition}

\begin{proof}  Any
completely isometric Jordan
homomorphism $i : A \to B$  into a $C^*$-algebra $B$ generated by $i(A)$,
extends by the uniqueness of the unitization (see Proposition \ref{ununau}) to a unital
completely isometric Jordan
homomorphism $i^1 : A^1 \to B^1$.   If $\theta : B^1 \to C^*_e(A^1)$ is the $*$-homomorphism
coming from the universal property of $C^*_e(A^1)$ (see e.g.\ \cite[Theorem 4.3.1]{BLM}),
then  $\theta_{\vert B} : B \to D$ is a surjective  $*$-homomorphism with  $j_{\vert A} = \theta \circ i$.
It also follows that $j(A)$  is a Jordan subalgebra of $C^*_e(A^1)$, and $j_{\vert A}$ is an approximately unital Jordan isomorphism onto $j(A)$.  \end{proof}

\subsection{Cohen factorization for Jordan modules}
The Cohen factorization theorem is a crucial tool for Banach and operator algebras,
and their modules.   In this section we prove a variant that works for 
Jordan operator algebras and their `modules'.

 Let $A$ be a Jordan operator algebra.  A Banach space (resp.\ operator space) $X$ together with a
contractive  (resp.\ completely contractive) bilinear map $A \times X \to X$ is called a left {\em Jordan Banach  (resp.\ operator) $A$-premodule}.
If $A$ is approximately unital then we say that $X$ is nondegenerate
if $e_t x \to x$  for $x \in X$, where $(e_t)$ is
a cai for $A$ (in this section when we say `cai' we mean
`partial cai'. It will follow from the next theorem 
that if one cai for $A$ works here then so will any other cai).
Similar definitions hold in the `right premodule' case, and a {\em Jordan Banach  (resp.\ operator) $A$-prebimodule}
is both a left and a right   Jordan Banach  (resp.\ operator) $A$-premodule 
such that $a (xb) = (ax)b$ for all $a,b \in A, x \in X$.
We remark that this definition is not related to the classical notion of a  Jordan module due to 
Eilenberg (cf.\ \cite[p.\ 512]{Jac}) .
A good example to bear in mind is the case where $X = C^*_e(A)$ or $X = C^*_{\rm max}(A)$.

If $X$ is a nondegenerate Jordan Banach $A$-premodule 
(resp.\ $A$-prebimodule) then $X$ is a Jordan Banach $A^1$-premodule (resp.\ $A^1$-prebimodule) 
 for the natural unital `action'.
For example,
 if $b = a + \lambda 1
 \in A^1$, and $x \in X$, then 
$\Vert (e_t \circ a) x - ax \Vert \leq \Vert x\Vert \Vert e_t \circ a - a \Vert \to 0,$
 and so 
$$\Vert (e_t \circ b) x- bx \Vert \leq \Vert (e_t \circ a) x- ax \Vert + \Vert
\lambda (e_t x- x) \Vert \to 0.$$ 
Hence $\Vert b x\Vert = \lim_t \, \Vert  (e_t \circ b) x\Vert
\leq \Vert x\Vert \Vert e_t \circ b \Vert \leq \Vert x\Vert \Vert b \Vert.$
Similarly $\Vert x b \Vert \leq \Vert x\Vert \Vert b \Vert$  in the prebimodule case.  

 We say that such $X$ is 
a  (left) {\em Jordan Banach $A$-module} if further if $(a_1 a_2 \cdots a_m) x = a_1 (a_2( \cdots (a_m x) ) \cdots )$ for all $m \in \Ndb$ and
$a_1, \cdots , a_m \in A^1$ such that $a_1 a_2 \cdots a_m \in A^1$.   The latter product 
is the one on some fixed containing $C^*$-algebra, for example $C^*_{\rm max}(A^1)$.
In fact we shall only need the cases $m = 2, 3$ in
the results below.  
Similar notation holds in the bimodule case, or 
for Jordan operator $A$-modules and bimodules.

The condition $a (xb) = (ax)b$  often holds automatically:

\begin{proposition}  \label{ob} Suppose that $X$ is an operator space, $A$ is
an approximately unital Jordan operator algebra, and that there are 
completely contractive bilinear maps $A \times X \to X$ and $X \times A \to X$
which are nondegenerate in the sense that
 $e_t x \to x$ and $x e_t \to x$ for $x \in X$, where $(e_t)$ is
a cai for $A$.   Then $a (xb) = (ax)b$ for all $a,b \in A, x \in X$.
\end{proposition}

\begin{proof}   The `actions'
 are oplications in the sense of \cite[Theorem 4.6.2]{BLM},
and by that theorem there are linear complete contractions
$\theta : A \to {\mathcal M}_l(X)$ and $\pi : A \to {\mathcal M}_r(X)$ such that $a x
= \theta(a)(x)$ and $x b = \pi(b) x$ for all $a,b \in A, x \in X$.
Since left and right multipliers commute (see 4.5.6 in \cite{BLM}), $(ax) b = a(xb)$ for such
$a,b, x$.  
 \end{proof}

\begin{remark}  The last proof shows that all `nondegenerate' Jordan operator prebimodules occur
via linear complete contractions
$\theta : A \to {\mathcal M}_l(X)$ and $\pi : A \to {\mathcal M}_r(X)$ such that $a x
= \theta(a)(x)$ and $x b = \pi(b) x$ for all $a,b \in A, x \in X$.
This should lead to a good theory of `dual Jordan operator bimodules'
similar to e.g.\ p.\ 183 in \cite{BLM}.
\end{remark}
\medskip

The following is a Jordan algebra version of the Cohen factorization
theorem:

\begin{theorem} \label{coh}  If $A$ is an
approximately unital Jordan operator algebra,
and if $X$ is a nondegenerate Jordan Banach $A$-module
(resp.\ $A$-bimodule), and if $b \in X$ then 
there exists an element $b_0 \in
X$ and an element $a \in {\mathfrak F}_A$ with $b = a b_0$ (resp.\ $b = a b_0 a$).
Moreover if $\Vert b \Vert < 1$ then $b_0$ and $a$ may be chosen
of norm $< 1$.  Also, $b_0$ may be chosen 
to be in the closure of $\{ a b  : a \in A \}$ (resp.\  $\{ a b a : a \in A \}$).   
\end{theorem} \begin{proof}  We follow the usual 
Cohen method as in the proof of e.g.\ 4.4 and 4.8  of 
\cite{BOZ}.     
Suppose that $b \in X$ with $\Vert b \Vert < 1$.
Given any $\varepsilon>0$, let $a_0=1$. Choose $f_1 \in \frac{1}{2}{\mathfrak F}_A$ from 
the cai such that 
$$\Vert (b a_0^{-1}) (1-f_1)\Vert+\Vert(1-f_1) (a_0^{-1} b) \Vert< 2^{-2}\varepsilon.$$ 
Let $a_1=2^{-1}f_1+2^{-1},$ then $a_1\in {\mathfrak F}_{A^1}$.
By the Neumann lemma $a_1$ is invertible in oa$(1,a_1)$,
and has inverse in $A^1$ with
 $\Vert a_1^{-1}\Vert\leq 2.$ Similarly, choose  $f_2\in \frac{1}{2}{\mathfrak F}_A$ such that 
$$\Vert ( b a_1^{-1}) (1-f_2)\Vert+\Vert(1-f_2)(a_1^{-1} b) \Vert< 2^{-4}\varepsilon.$$
By induction, let $a_n=\sum_{k=1}^n \,  2^{-k}f_k+2^{-n}$.
We have $$\Vert 1-a_n\Vert=\Vert\sum_{k=1}^n2^{-k}(1-f_k)\Vert\leq \sum_{k=1}^n 2^{-k}=1-2^{-n}.$$
By the Neumann lemma $a_n$ is invertible in oa$(1,a_n)$,
and has inverse in $A^1$ with 
$\Vert a_n^{-1}\Vert\leq2^n.$   
 Choose $f_{n+1}\in \frac{1}{2}{\mathfrak F}_A$ such that 
$$\Vert ( b  a_n^{-1}) (1-f_{n+1})\Vert+\Vert(1-f_{n+1}) (a_n^{-1} b) \Vert< 2^{-2(n+1)}\varepsilon.$$
Note that  $a_{n+1}^{-1} - a_n^{-1} = a_n^{-1} (a_n - a_{n+1}) a_{n+1}^{-1}
= 2^{-n-1} a_n^{-1} (1-f_{n+1})  a_{n+1}^{-1}$, and similarly  $a_{n+1}^{-1} - a_n^{-1} = 2^{-n-1} a_{n+1}^{-1} (1-f_{n+1})  a_{n}^{-1}$.
Set $x_n=a_n^{-1}b$   (resp.\ $x_n=a_n^{-1}ba_n^{-1}$).   We continue in the bimodule case, the left module 
case is similar but easier.   We have  
\begin{align*} x_{n+1}-x_n &= a_{n+1}^{-1}b a_{n+1}^{-1}-a_n^{-1}ba_n^{-1} 
= a_{n+1}^{-1}b (a_{n+1}^{-1}-a_n^{-1}) + (a_{n+1}^{-1} - a_n^{-1}) b a_n^{-1}   \\
&= 2^{-n-1} (a_{n+1}^{-1}b (a_{n}^{-1} (1-f_{n+1})  a_{n+1}^{-1}) + (a_{n+1}^{-1} (1-f_{n+1})  a_{n}^{-1}) ba_n^{-1} ) . \end{align*} 
Because of the relations $x (abc) = ((xa)b)c$ and $(abc) x = a(b(cx))$,  $\Vert x_{n+1}-x_n\Vert$ is dominated by 
\begin{align*}
 & 2^{-n-1} (\Vert a_{n+1}^{-1} \Vert \Vert (b  a_n^{-1})  (1-f_{n+1}) \Vert \Vert  a_{n+1}^{-1}
\Vert + \Vert a_{n+1}^{-1}   \Vert \Vert (1-f_{n+1})  (a_{n}^{-1} b)
\Vert \Vert a_{n}^{-1} \Vert) \\
& \leq 2^{-n-1} (\Vert a_{n+1}^{-1} \Vert^2  + \Vert a_{n+1}^{-1} \Vert \Vert a_{n}^{-1} \Vert) 2^{-2(n+1)}\varepsilon
 \\
& \leq   2^{-3(n+1)} (2^{2(n+1)} + 2^{2n+1})  \varepsilon < 2^{-n}\varepsilon. \end{align*} 

Therefore, $\{x_n\}$ is a Cauchy sequence in $X$. Let $b_0=\lim_n x_n$ and $a=\sum_{k=1}^{+\infty} 2^{-k}f_k,$ then $a\in \frac{1}{2}{\mathfrak F}_A$. Hence,
$b=ab_0a$ since $b=a_n x_n a_n$ and $a_n\to a$ and $x_n\to b_0.$ 
Also, 
\begin{align*}
	\Vert x_n-b\Vert\leq \sum_{k=1}^n\Vert x_k-x_{k-1}\Vert\leq 2\varepsilon,
\end{align*}
so that $\Vert b-b_0\Vert\leq 2\varepsilon.$   Thus $\Vert b_0 \Vert \leq \Vert b \Vert + 2\varepsilon$, 
and this is $< 1$ if $2\varepsilon < 1 - \Vert b \Vert $.
Choose some $t>1$ such that $\Vert tb\Vert<1$. By the argument above, there exists $a\in \frac{1}{2} {\mathfrak F}_A$ and $b_0\in B$ of norm $< 1$ such that $tb=ab_0a$.  Let $a'=\frac{a}{\sqrt{t}}$, then $b=a'b_0a'.$ Then  $\Vert a'\Vert <1$ and $\Vert b_0\Vert<1.$
\end{proof}

\begin{corollary} If $A$ is an 
approximately unital Jordan operator subalgebra of a $C^*$-algebra $B$, 
and if $B$ is generated as a $C^*$-algebra by $A$, then if $b \in B$ 
there exists an element $b_0 \in
B$ and an element $a \in {\mathfrak F}_A$ with $b = a b_0 a$.
Moreover if $\Vert b \Vert < 1$ then $b_0$ and $a$ may be chosen
of norm $< 1$.  Also, $b_0$ may be chosen 
to be in the closure of $\{ a b a : a \in A \}$.
\end{corollary}   

\begin{proof}  This follows immediately from the Cohen type Theorem \ref{coh}  above.  \end{proof}

There is a similar one-sided result
using the one-sided version of our Cohen factorization result.

 If $A$ is a Jordan operator subalgebra of a $C^*$-algebra $B$ then we say that  a net $(e_t)$ in 
Ball$(A)$  is 
a {\em left $B$-partial cai} for $A$ if  $e_t a \to a$   for all $a \in A$.  Here we are using the usual
product on $B$, which may not give an element in $A$, and may depend on $B$.
We then can factor any $b \in C^*_B(A)$ as $a b_0$ for $a, b_0$ as above,
using the one-sided Cohen factorization result above. 
We remark that by a modification of the proof of
 Lemma \ref{jcai} and Theorem \ref{kapl} one can show that 
the following are equivalent:
\begin{itemize} \item [(i)] $A$ has a left $B$-partial cai.
 \item [(ii)]   $A^{**}$ has a left identity $p$ of norm 1 with respect to  the usual  product
in $B^{**}$.    
\item [(iii)]   $A$ has a left $B$-partial cai in $\frac{1}{2} {\mathfrak F}_A$.
\end{itemize}  If these 
hold then $p$ is an open projection in $B^{**}$ in the sense 
of Akemann {\rm \cite{Ake2}} (that is, is the weak* limit of an increasing net in $B$).

\subsection{Jordan representations}   Following the (associative) operator algebra case we have:

\begin{lemma}\label{2deg}
Let $A$ be an approximately unital Jordan operator algebra and let  $\pi : A \to B(H)$ be a 
contractive Hilbert space Jordan  representation.  We let $P$ be the projection onto $K = [\pi(A)H]$. Then
$\pi(e_t)\to P$ in the weak* (and WOT) topology of $B(H)$ for any J-cai $(e_t)$ for $A$.   Moreover, for
$a \in A$ we have $\pi(a) \; = \; P\pi(a)P$, and the compression of $\pi$ to $K$ is
a contractive Hilbert space Jordan  representation.    Also, if $(e_t)$ is a partial cai for $A$, then 
 $\pi(e_t) \pi(a) \to \pi(a)$ and   $\pi(a) \pi(e_t)  \to \pi(a)$.
In particular, $\pi(e_t)_{|K} \to I_K$ SOT in $B(K)$.    
\end{lemma}

\begin{proof}
Let $A$ be an approximately unital Jordan operator algebra and let  $\pi : A \to B(H)$ be a 
contractive Hilbert space Jordan  representation.   If $(e_t)$ is a J-cai for $A$ then by the proof
of Lemma \ref{jcai} we have that $e_t \to p$ weak*, where $p$ is an identity for $A^{**}$.
The canonical weak*
continuous extension $\tilde{\pi} : A^{**} \to B(H)$ takes $p$ to a projection $P$ on $H$,
and $\pi(e_t) \to P$ WOT.    Note that $$\tilde{\pi}(p) \pi(a) + \pi(a)
\tilde{\pi}(p) = P \pi(a) + \pi(a) P = 2 \pi(a), \qquad a \in A, $$
so that as in the proof of  Lemma \ref{jcai} we have $P \pi(a) = \pi(a) P =  \pi(a)$ for $a \in A$.
We have $\pi(e_t) \pi(a) \to \tilde{\pi}(p) \pi(a) = \pi(a)$ WOT.     If $Q$ is
the projection onto $[\pi(A)H]$ it follows that $P Q = Q$, so $Q \leq P$.
If $\eta \perp \pi(A)H$ then $0 = \langle \pi(e_t) \zeta , \eta \rangle
\to  \langle P  \zeta , \eta \rangle$, so that $\eta \perp P(H)$.   Hence $P(H)
\subset [\pi(A)H]$ and so $P \leq Q$ and $P = Q$.     It is now evident that the compression of $\pi$ to $[\pi(A)H]$ is
a contractive Hilbert space Jordan  representation. 

Suppose that $\rho : A \to C^*_{{\rm max}}(A)$ is the canonical map.
If $(e_t)$ is a partial cai for $A$, then  $\rho(e_t) \rho(a) \to \rho(a)$ and $\rho(a) \rho(e_t) \to \rho(a)$.
If $\theta : D  = C^*_{\rm max}(A) \to B(H_\pi)$ is the $*$-homomorphism with
$\theta \circ \rho = \pi$ then $$\pi(e_t) \pi(a) =
\theta(\rho(e_t) \rho(a)) \to \theta(\rho(a)) = \pi(a).$$
Similarly $\pi(a) \pi(e_t)  \to \pi(a)$.
\end{proof}

Define a {\em nondegenerate Jordan representation} of a 
  Jordan operator algebra on $H$ to be a contractive Hilbert space Jordan  representation $\pi : A \to B(H)$ 
such that $\pi(A)H$ is dense in $H$.   By the last result the canonical weak*
continuous extension $\tilde{\pi} : A^{**} \to B(H)$ is unital iff $\pi$ is nondegenerate.  Indeed 
$\tilde{\pi}(1) = I_H$ iff $\pi(e_t) \to I_H$ weak* in $B(H)$ for any partial cai $(e_t)$ of 
$A$, that is
iff $\pi(e_t) \to I_H$ WOT.  

 Let $H, K, \pi$ be as in Lemma \ref{2deg}. If we regard $B(K)$ as a
subalgebra of $B(H)$ in the natural way (by identifying any $T$ in $B(K)$ with the map
$T\oplus 0$ in $B(K\oplus K^{\perp}) =B(H)$), then the
Jordan   homomorphism $\pi$ is valued in $B(K)$. Note that  $\pi$ is
nondegenerate when regarded as valued in $B(K)$, since $\pi(e_t) \pi(a) \to  \pi(a)$ WOT.  
As in the (associative) operator algebra case \cite{BLM},
this yields a
principle whereby to reduce a possibly degenerate Jordan  homomorphism to
a nondegenerate one.  One corollary is that for any approximately unital Jordan operator algebra
$A$, there exist a
Hilbert space $H$ and a nondegenerate completely isometric
Jordan homomorphism $\pi\colon A\to B(H)$.

As an application, we may see using Lemma \ref{jcai}
 that if $B$ is a $C^*$-cover of an
approximately unital Jordan operator algebra $A$,
and if $\pi\colon B\to B(H)\,$ is a
$*$-representation, then $\pi$ is nondegenerate if and only if its
restriction $\pi_{\vert A}$ is nondegenerate.  
For if $\pi$ is nondegenerate then $\pi(e_t) \to I$ WOT
where $(e_t)$ is a partial cai for $A$, since then $(e_t)$ is a
cai for $B$.

\subsection{Approximate identities and functionals}  \label{aifnl}  

Following \cite[Proposition 2.1.18]{BLM} we have:

\begin{lemma} \label{extunfu}  Let $A$ be an approximately unital Jordan operator algebra with a  partial cai $(e_t)$.
Denote the identity of $A^1$ by $1$
\begin{enumerate}
\item [(1)] If $\psi: A^1 \rightarrow \Cdb$ is a functional on $A^1$, then $\lim_t \psi(e_t)=\psi(1)$ if and only if $\|\psi\|=\|\psi_{\vert_A}\|$.

\item [(2)] Let $\varphi: A\rightarrow \Cdb$ be any functional on $A$. Then $\varphi$ uniquely extends to a functional on $A^1$ of the same norm.
\end{enumerate}
\end{lemma}

\begin{proof}
(1) \ Suppose that $\psi: A^1 \rightarrow \Cdb$ satisfies  $\lim_t \psi(e_t)=\psi(1).$ For any $a\in A$ and $\lambda\in \Cdb$, we have
 $\lim_t\psi(a  \circ e_t +\lambda e_t)=\psi(a+\lambda1)$, and so 
$$|\psi(a+\lambda1)|\leq \Vert\psi_{\vert_A}\Vert \lim_t\Vert a \circ e_t+\lambda e_t \Vert\leq \Vert\psi_{\vert_A}\Vert \sup_t\Vert
a \circ e_t+\lambda e_t \Vert =  \|\psi_{\vert_A}\| \|a+\lambda1\|.$$ 
(We have used Proposition \ref{ununau}.)   Hence
 $\|\psi\|=\|\psi_{\vert_A}\|.$   

Conversely, suppose that $\|\psi\|=\|\psi_{\vert_A}\|$, which we
may assume to be $1$.  We may extend $\psi$ to $C^*(A^1)$,
and then there exists a unital $*$-representation $\pi :
C^*(A^1) \to B(H)$ and vectors $\xi,\eta \in {\rm Ball}(H)$ with
$\psi(x)=\langle \pi(x)\xi,\eta\rangle$ for any $x\in A^1.$
Let $K=[\pi(A)\xi]$, and let $p$ be the projection onto $K$. For any $a\in A$, we have $\langle\pi(a)\xi,\eta\rangle=\langle p\pi(a)\xi,\eta\rangle$, and so $$|\langle \pi(a)\xi,\eta\rangle|=|\langle\pi(a)\xi,p\eta\rangle|\leq \|\pi(a)\xi\|\|p\eta\|\leq \|a\|\|p\eta\|.$$
This implies that $1 = \|\psi_{\vert_A}\|\leq \|p\eta\|$, so that 
$\eta \in K$.   By Lemma \ref{2deg} we have that $(\pi(e_t))$ 
converges WOT to the projection onto $[\pi(A) H]$, and so 
$$\psi(e_t)=\langle \pi(e_t)\xi,\eta\rangle \rightarrow \langle \xi, p \eta\rangle =\langle \xi, \eta \rangle =\psi(1).$$
 
(2) \ If $\varphi \in A^*$ then similarly to the above
there exists a nondegenerate $*$-representation $\pi :
C^*(A) \to B(H)$ and vectors $\xi,\eta \in {\rm Ball}(H)$ with
$\psi(x)=\langle \pi(x)\xi,\eta\rangle$ for any $x\in A$.
We have $\psi(e_t) = \langle \pi(e_t) \xi,\eta\rangle
\to \langle \xi,\eta\rangle$.   We may now finish as in
the proof of \cite[Proposition 2.1.18 (2)]{BLM}.  \end{proof}

Define a state on an approximately unital Jordan operator algebra to be a functional 
satisfying the conditions in the next result.

\begin{lemma} \label{eqsta}
For a norm $1$ functional $\varphi$ on an approximately unital Jordan operator algebra $A$, the following are equivalent:
\begin{enumerate}
	\item [(1)] $\varphi$ extends to a state on $A^1$.
	\item [(2)]  
	$\varphi(e_t)\rightarrow 1$ for every partial cai for $A$.
	\item [(3)]  
	$\varphi(e_t)\rightarrow 1$ for some partial cai for $A$.
	\item [(4)]  
	$\varphi(e)=1$ where $e$ is the identity of $A^{**}$.
	\item [(5)] 
	$\varphi(e_t)\rightarrow 1$ for every Jordan cai for $A$.
	\item [(6)] 
	$\varphi(e_t)\rightarrow 1$ for some Jordan cai for $A$.
	\end{enumerate}	
\end{lemma}
\begin{proof}  That  (1) $\Rightarrow$ (2) follows from Lemma \ref{extunfu}.
That (3) implies (4), (6) $\Rightarrow$ (1),
and (4) $\Leftrightarrow$ (5) \ follows from the last assertion of Lemma \ref{jcai},
 that any Jordan cai for $A$ converges to $1_{A^{**}}$.
Clearly, (2) implies (3), and  (5) implies (6).   \end{proof}

\begin{corollary} \label{injov}  Let $A$ be an approximately unital Jordan operator algebra.   Then any injective 
envelope of $A^1$ is an injective 
envelope of $A$.   Moreover this may be taken to be a unital
$C^*$-algebra $I(A)$  
containing $C^*_e(A)$ as a $C^*$-subalgebra, and hence containing 
$A$ as a Jordan subalgebra.   Finally,  $C^*_e(A)$ is the ternary envelope ${\mathcal T}(A)$
of $A$ in the sense of {\rm \cite{Ham}} or 
{\rm \cite[Section 8.3]{BLM}}.  \end{corollary}

\begin{proof}
This follows just as in \cite[Corollary 4.2.8 (1) and (2)]{BLM},
but appealing to Lemma \ref{extunfu} above in place of the
reference to 2.1.18 there.   Note that
$I(A^1)$ may be taken to be a unital
$C^*$-algebra containing $C^*_e(A^1)$ 
 as a unital $C^*$-subalgebra.   The proof of Proposition \ref{3auce} 
shows that the inclusion $A^1 \to  C^*_e(A^1)$ is a Jordan morphism,
so $I(A^1)$ contains  $C^*_e(A)$ as a $C^*$-subalgebra, 
and $A$ as a Jordan subalgebra.     For the last part note that products in $I(A)$ of the form 
$a_1 a_2^* a_3 \cdots a_{2n}^* a_{2n+1}$ for $a \in A$, lie in $C^*_e(A)$.
Conversely any finite product of alternating terms  in $A$ and $A^*$ 
may be viewed as a limit of such products beginning and ending with a term in $A$,
using  a partial cai  $(e_t)$ for $A$ and  \cite[Lemma 2.1.6]{BLM}.  So $C^*_e(A) = {\mathcal T}(A)$.
\end{proof}

It follows as in the introduction to \cite{BRord}  that states on $A$ are also the norm 1 functionals that extend to a 
state on any containing $C^*$-algebra generated by $A$.  

It follows from facts in Section 
\ref{MeyerRP} that for any Jordan operator algebra $A$,  $x \in {\mathfrak r}_A$  iff ${\rm Re}(\varphi(x)) \geq 0$ for all
states $\varphi$ of $A^1$.   Indeed, such $\varphi$
extend to states on $C^*(A^1)$.     

\subsection{Multiplier algebras}

Let $A$ be an approximately unital Jordan operator algebra and let $(C^*_e(A),j)$ be its
$C^*$-envelope.   Let $i : A \to B$ be a completely isometric Jordan morphism into a $C^*$-algebra.
Suppose that $a, b \in A$ and that $i(a) i(b) \in i(A)$.    If $\theta : C^*_B(i(A)) \to C^*_e(A)$ is 
the $*$-homomorphism coming from the universal property then $j(a) j(b) = \theta(i(a) i(b))  \in \theta(i(A)) = j(A)$,
and $$j^{-1}(j(a) j(b)) = j^{-1}(\theta(i(a) i(b))) = i^{-1}(i(a) i(b)).$$  
This shows that the `product' in $B$ of elements in $A$, if it falls in $A$,
 matches the product in $C^*_e(A)$.   
With this in mind we define the left multiplier algebra $LM(A)$ to be the set 
$\{\eta \in A^{**} : \eta A\subset A\},$  where the
product here is the one in $C^*_e(A)^{**}$.  We will soon see that this is in fact an (associative) algebra.
 We define the right multiplier algebra  $RM(A)$ and multiplier algebra $M(A)$ analogously.  If $A$ is unital then
these algebras are contained in $A$.

\begin{lemma} \label{pin}   Let $A$ be an approximately unital Jordan operator algebra.
If $p$ is a projection in $LM(A)$ then $p \in M(A)$.   More generally, the diagonal $\Delta(LM(A)) \subset M(A)$.   
\end{lemma}

\begin{proof}   See  \cite[Lemma 5.1]{BHN} for the operator algebra case. 
Let $(e_t)$ be a partial cai  for $A$. 
In $C^*_e(A)^{**}$ we have 
by \cite[Lemma 2.1.6]{BLM} that
 $$ap = \lim_t a e_t^* p = \lim_t \, a (p^*e_t)^*  \in C^*_e(A),  \qquad a \in A.$$  Also $a \circ p \in A^{\perp \perp}$ so 
$ap = 2 a \circ p - pa \in A^{\perp \perp}$.
So $ap \in A^{\perp \perp} \cap C^*_e(A) = A$, and hence $p \in M(A)$.   The same proof works if 
$p \in \Delta(LM(A))$.   \end{proof}

 \begin{theorem} \label{tlma}
Let $A$ be an approximately unital Jordan operator algebra and let
$B = C^*_e(A)$, with $A$ considered as a Jordan  subalgebra. 
Then $LM(A) = \{\eta\in B^{**} :  \eta A\subset A\}$.   This is  
completely isometrically isomorphic to the (associative) operator algebra ${\mathcal M}_\ell(A)$ of
operator space left multipliers of $A$
in the sense of e.g.\ {\rm \cite[Section 4.5]{BLM}}, and is completely isometrically isomorphic to 
a unital subalgebra of $CB(A)$.   Also, $\Vert T \Vert_{\rm cb} = \Vert T \Vert$ for 
$T \in LM(A)$ thought of as an operator on $A$.   Finally, for any nondegenerate completely isometric Jordan representation $\pi$ of $A$ on a Hilbert space $H$, the algebra $\{T\in B(H): T\pi(A)\subset \pi(A)\}$ is completely isometrically isomorphic to 
a unital subalgebra of $LM(A)$, and this isomorphism maps onto  $LM(A)$ if $\pi$ is a faithful nondegenerate  $*$-representation of $B$ (or a nondegenerate completely isometric  representation of ${\rm oa}_e(A)$). 
 \end{theorem}

\begin{proof}  
Obviously $LM(A) \subset \{\eta \in B^{\ast\ast} : \eta A\subset A\}.$  
Conversely, if $\eta$ is in the latter set then $\eta e_t\in A,$ where $(e_t)$ is partial cai for $A.$ Hence $\eta\in A^{\ast\ast}$,
since by Lemma \ref{jcai}  $(e_t)$ is a cai for $B$.   So $LM(A) = \{\eta\in B^{**} :  \eta A\subset A\}$.

  Recall from Corollary \ref{injov} that
$I(A)$ is a unital
$C^*$-algebra.   It follows from 4.4.13 and the 
proof of Theorem 4.5.5 in \cite{BLM}
that ${\mathcal M}_\ell(A)$ is completely
isometrically isomorphic to
$\{ T \in I(A) : T j(A) \subset j(A) \}$. 
Note that                                        
$$T j(a)^* = \lim_t T j(e_t) j(a)^* \in j(A) j(A)^* \subset C^*_e(A).$$
Hence $T \in LM(C^*_e(A))$, and we may view     
${\mathcal M}_\ell(A)$ as    
$\{ T \in LM(C^*_e(A)) : T j(A) \subset j(A) \}$.    
If $\eta \in C^*_e(A)^{**}$ and     
$\eta j(A) \subset j(A)$ then as in the last centered formula     
and the line after it, we have $\eta \in LM(C^*_e(A))$.   
So ${\mathcal M}_\ell(A) \cong
\{ \eta \in C^*_e(A)^{**} : T j(A) \subset j(A) \}$.     
Thus from the last paragraph
$LM(A) \cong {\mathcal M}_\ell(A)$.    
We remark that this may also be deduced from e.g.\ 8.4.1 in \cite{BLM}. It also follows that for any $u \in {\mathcal M}_\ell(A)$,
${\rm w^*lim}_t \, u(e_t)$ exists in $A^{**}$, and equals $\sigma(u)$ where 
$\sigma : {\mathcal M}_\ell(A) \to LM(A)$ is the isomorphism above.

The canonical map $L : LM(A) \to CB(A)$ is a completely contractive homomorphism.   On the other hand 
for $[\eta_{ij} ] \in M_n(LM(A))$  we have $$\Vert L(\eta_{ij}) ] \Vert_{M_n(CB(A))}
\geq \Vert [ \eta_{ij} e_t ] \Vert .$$    It follows by Alaoglu's theorem that in the weak* limit with $t$, 
$\Vert [ \eta_{ij} ] \Vert \leq \Vert L(\eta_{ij}) ] \Vert_{M_n(CB(A))}$.
Thus  $LM(A)$ is completely isometrically isomorphic to 
a unital subalgebra of $CB(A)$.     Note that $$\| [\eta_{ij} a_{kl} ] \|
= \lim_t \| [\eta_{ij} e_t  a_{kl} ] \| \leq \sup_t \| [\eta_{ij} e_t ] \|, \qquad
 [ a_{kl} ] \in {\rm Ball}(M_m(A)),$$ so the last supremum equals the cb norm 
of $[\eta_{ij}]$ thought of as an element of $M_n(CB(A))$.

Let $LM(\pi) = \{T \in B(H): T\pi(A)\subset \pi(A)\}$.   There is 
a canonical complete contraction $LM(\pi) \to {\mathcal M}_\ell(A)$.
Composing this with the map $\sigma : {\mathcal M}_\ell(A) \to LM(A)$ above gives a
homomorphism $\nu : LM(\pi) \to LM(A)$.    
The canonical weak* continuous extension 
$\tilde{\pi} : A^{**} \to B(H)$ is a completely contractive Jordan homomorphism,
and $$\tilde{\pi} ( \nu(T)) = {\rm w^*lim}_t \pi(\pi^{-1}(T \pi(e_{t}))) = T,  \qquad T \in LM(\pi),$$
by the  nondegeneracy of $\pi.$
It follows that $\nu$ is completely isometric.

If $\pi$ is a faithful nondegenerate  $*$-representation of $B$ or a nondegenerate completely isometric  representation of ${\rm oa}_e(A)$, and $T \in B(H)$ with 
$T \pi(A) \subset \pi(A)$ then as in the second paragraph of the proof above we have $T \pi(B) \subset \pi(B)$ or 
$T \pi({\rm oa}_e(A)) \subset \pi({\rm oa}_e(A))$.   Thus  in the first case we may identify 
$LM(\pi)$ with $\{ \eta \in B^{**} : \eta A \subset A \}$, which we saw above was $LM(A)$.
A similar argument works  in the second case. 
\end{proof}

\begin{definition} \label{defjma} If $A$ is an
approximately unital  Jordan operator algebra, the {\em Jordan multiplier algebra} of $A$ is
$$JM(A)=\{\eta\in A^{**}: \eta a+a\eta\in A, \forall a\in A\}.$$
This is a unital Jordan operator algebra in which $A$ is
an approximately unital Jordan ideal, follows by using the identity (\ref{abc})  
in the obvious computation). 
 \end{definition}

\begin{remark} \label{adefjma} Presumably there is also a variant of this definition
in terms of operators in $B(H)$, if $A \subset B(H)$ nondegenerately.
\end{remark}

Note that $A = JM(A)$ if $A$ is unital.
If a projection $p \in A^{**}$ is in 
$JM(A)$ then $p A p \subset A$.   This follows from the identity (\ref{aba}).
      Of course $M(A) \subset JM(A)$.  

 \section{Hereditary subalgebras, ideals, and open projections} 

\subsection{Hereditary subalgebras and open projections}  \label{HSA} 
Through this section $A$ is a Jordan operator algebra (possibly not approximately unital). Then $A^{**}$ is a Jordan operator algebra. We say that a projection in $A^{**}$ is {\em open} in $A^{**}$, or 
$A$-{\em open} for short,
if $p\in (pA^{**}p\cap A)^{\perp\perp}$.  That is, iff there is a net $(x_t)$ in $A$ with
$$x_t = p x_t p \to p \; \; {\rm weak}^* .$$
This is a derivative of Akemann's notion of open projections for $C^*$-algebras, 
a key part of his powerful variant of noncommutative topology (see e.g.\
\cite{Ake2,P}).    If $p$ is open in $A^{**}$ then clearly $$D=pA^{**}p\cap A =\{a\in A: a=pap \}$$ is a closed Jordan subalgebra of $A$, and  the Jordan subalgebra $D^{\perp \perp}$ of $A^{**}$ has identity $p$ (note $x_t \in D$).
By Lemma \ref{jcai} $D$ has a partial cai (even one in $\frac{1}{2} {\mathfrak F}_A$ by Theorem \ref{kapl}).    If $A$ is also
approximately unital then a projection $p$  in $A^{**}$ is $A$-closed if $p^\perp$ is $A$-open.  
 
We call such a Jordan subalgebra $D$ a {\em hereditary subalgebra} (or HSA for short) of $A$, and we say that $p$ is the {\em support projection} of $D$.  It follows from the above that  the support projection of a HSA is the weak* limit of any partial cai from the HSA.   
One consequence of this is that a projection in $A^{**}$ is open in $A^{**}$
iff it is open in $(A^1)^{**}$.

We remark that if $A$ is a JC*-algebra then the net $(x_t)$ above may be taken to be
positive in the definition of hereditary subalgebra, or of open projections, and in many of the results below one may provide `positivity proofs' as opposed to 
working with real positive elements.   We leave this case of the theory to the 
interested reader.
 
\begin{corollary} \label{iffsu}  For any Jordan operator algebra $A$, a projection $p \in A^{**}$ is  $A$-open iff
$p$ is the  support projection of a HSA in $A$.  \end{corollary}

\begin{proposition}
\label{mpoc}
For any approximately unital Jordan operator algebra $A$, every
projection $p$ in the Jordan multiplier algebra $JM(A)$ 
(see {\rm \ref{defjma}}) is $A$-open and also $A$-closed.  
\end{proposition}
\begin{proof}
Indeed, if $A$ is approximately unital and $(e_t)$ is a partial cai of $A$, then $pA^{**}p\cap A=pAp$ by Remark \ref{adefjma},
and $(pe_tp) \subset A$ has weak* limit $p$.  So $p$ is $A$-open. Similarly, $p^{\perp}$ is $A$-open since $p^{\perp}\in JM(A).$
\end{proof}

If $B$ is a $C^*$-algebra containing a Jordan operator algebra $A$
and $p$ in $A^{**}$ is  $A$-open then $p$ is 
open as a projection in $B^{**}$ (since it is the weak* limit of 
a net $(x_t)$ with $x_t = p x_t p$, see \cite{BHN}).
Unfortunately at the time of writing this paper we did not yet know if the converse of this,
Hay's theorem \cite{Hay}, was true 
for Jordan operator algebras $A$.  That is  we did not  know if 
 $p$ is
open in $B^{**}$, and 
$p \in A^{\perp \perp}$,  implied $p$ is  $A$-open.    Subsequently this obstacle was overcome and
will appear in \cite{BNj}.
This Jordan variant of Hay's theorem could be used to simplify the proofs of a few results in the present paper, namely one could in those places use the proofs from \cite{BHN,BRI, BRII}.
We have chosen to leave the original proofs since the proof of the Jordan variant of Hay's theorem is deep,
but mainly because the present proofs may be more relevant in categories that are further from $C^*$-algebras, 
such as certain classes of Banach algebras.
 
We recall a Jordan subalgebra $J$ of a Jordan operator algebra $A$ is an {\em inner ideal} (in the Jordan sense) if for any $b,c\in J$ and $a\in A$, then $bac+cab\in J$  (or equivalently, $bAb \subset J$ for all $b \in J$
as we said in the introduction where we defined HSA's).
  
\begin{proposition} \label{hsaii}
A subspace $D$ of a Jordan operator algebra $A$ is a HSA if and only if it is an approximately unital inner ideal in the Jordan sense.   In this case $D^{\perp\perp}=pA^{**}p$,  where $p$ is the support projection of the HSA.
Conversely if $p$ is a projection in $A^{**}$ and $E$ is a subspace of $A^{**}$ such  that $E^{\perp \perp} =
pA^{**}p$, then $E$ is a HSA and $p$ is its $A$-open support projection.
\end{proposition}

\begin{proof}
If $D$ is a HSA, with $D=\{b\in A :  pbp=b\}$, for some $A$-open projection $p\in A^{**}$, then for any $b,c\in D$ and $a\in A$, we have 
$$bac+cab=pbacp+pcabp=p(bac+cab)p.$$  
Hence  $bac+cab\in D$. Thus $D$ is an approximately unital inner ideal.

Conversely,  if $J$ is an approximately unital inner ideal, then $J^{\perp\perp}$ is  a Jordan operator algebra with identity $p$ say which is a weak* limit of a net in $J$. Clearly $J^{\perp\perp}\subseteq pA^{**}p.$ Conversely, by routine weak* density arguments $J^{\perp\perp}$ is an inner ideal, and so $J^{\perp\perp}=pA^{**}p$,
and $J = p A^{**}p \cap A$.  Hence $e$ is open and $J$ is an HSA. 
The last statement is obvious since by functional analysis
$E = pA^{**}p \cap A$.
\end{proof}

We also point out that if $D_1, D_2$ are HSA's in a Jordan operator algebra $A$, and if $p_1, p_2$ are their 
support projections, then $D_1 \subset D_2$ iff $p_1 \leq p_2$.   For the harder forward direction of this
note that $p_1 \in (D_1)^{\perp\perp} \subset (D_2)^{\perp\perp} = p_2 A^{**}p_2$, so that $p_1 \leq p_2$. 

Note that a closed Jordan ideal $J$ in an approximately unital Jordan operator algebra $A$,
 which possesses a partial cai, satisfies $a xb + bxa \in J$ for $a,b \in J, x \in A$ (that is, $J$ is
an inner ideal) by Proposition \ref{hsaii}.   So $J$ is a HSA.   

\begin{corollary}
Let $A$ be a Jordan operator algebra and $(e_t)$ is a net in ${\rm Ball}(A)$ such that $e_te_s\rightarrow e_t$ and $e_se_t\rightarrow e_t$ with $t$ (product in some $C^*$-algebra containing $A$). Then 
$$\{x\in A: xe_t\rightarrow x, e_tx\rightarrow x\}$$
is a HSA of $A$. Conversely, every HSA of $A$ arises in this way.
\end{corollary}
\begin{proof}
Denote $J=\{x\in A: xe_t\rightarrow x, e_tx\rightarrow x\}$, product in some $C^*$-algebra $D$ containing $A$.
It is easy to see that $J$ is a Jordan subalgebra of $A$ and $(e_t)$ is a $D$-relative 
partial cai of $A$.   So $J$ is approximately unital.  For any $x,y\in J$ and $a\in A$, then 
\begin{align*}
\Vert (xay+yax)e_t-(xay+yax)\Vert&=\Vert xa(ye_t-y)+ya(xe_t-x)\Vert \rightarrow 0.
\end{align*}
Similarly, $e_t(xay+yax)\rightarrow (xay+yax).$ Hence, $xay+yax\in J$. By Proposition \ref{hsaii}, $J$ is a
HSA.

Conversely, suppose that $D=pA^{**}p\cap A$, where $p$ is an $A$-open projection.
There exists a partial cai $(e_t)$ of $D$ with weak* limit $p.$  Denote 
$$J=\{x\in A: xe_t\rightarrow x, e_tx\rightarrow x\},$$ with product in some $C^*$-algebra containing $A$.
Then $J \subset D=\{x\in A: pxp=x\}.$   However clearly $D \subset J$.  
\end{proof}

As in  \cite[Theorem 2.10]{BHN} we have:

\begin{theorem}
Suppose that $D$ is a hereditary subalgebra of an approximately unital Jordan operator algebra $A$. Then every $f\in D^*$ has a unique Hahn-Banach extension to a functional in $A^*$ (of the same norm).   \end{theorem}
\begin{proof}  Follow the  proof of \cite[Theorem 2.10]{BHN}, 
viewing $A$ as a  Jordan subalgebra  of a $C^*$-algebra $B$, and working in $B^{**}$.   We note that 
there are  a few easily correctable typos in that proof.  If $p$ is the 
support projection of $D$ then 
since $A^{**}$ is a unital Jordan algebra we have $p \eta (1-p) + (1-p) \eta p \in  A^{**}$ for $\eta  \in  A^{**}$.
The argument for \cite[Theorem 2.10]{BHN} then shows that 
$g(p \eta (1-p) + (1-p) \eta p) = 0$ for any Hahn-Banach extension $g$ of $f$.   The rest of the proof is
identical.  
\end{proof}

The analogue of \cite[Proposition 2.11]{BHN} holds too.  For example, we have: 

\begin{corollary}
	Let $D$ be a HSA in an approximately unital Jordan operator algebra $A$.
Then any completely contractive map $T$ from $D$ into a unital weak* closed Jordan operator algebra $N$
such that $T(e_t) \to 1_N$ weak* for some partial cai $(e_t)$ for $D$, has a unique completely contractive  extension
$\tilde{T} : A \to N$ with $\tilde{T}(f_s) \to 1_N$ weak* for some (or all) partial cai $(f_s)$ for $A$.
\end{corollary}
\begin{proof}   The canonical weak* continuous extension $\hat{T} : D^{**} \to N$ is unital 
and completely  contractive, and can be extended to a weak* continuous unital complete contraction
$\Phi(\eta) = \hat{T}(p \eta p)$ on $A^{**}$, where $p$ is the support projection of $D$.   This in turn restricts to a  completely contractive 
$\tilde{T} : A \to N$ with $\tilde{T}(f_s) \to 1_N$ weak* for all partial cai $(f_s)$ for $A$.
For uniqueness,  any other such extension $T' : A \to N$ extends to a 
weak* continuous unital complete contraction $\Psi : A^{**} \to N$, and $\Psi(p) = \lim_t \Phi(e_t) = 1_N$.  Then $\Psi$ extends further 
to a unital completely positive $\hat{\Psi} : R \to B(H)$ where 
$R$ is a $C^*$-algebra containing $A^{**}$ as a unital subspace, and where $N \subset B(H)$ unitally.  
Then for $\eta \in A^{**}$ we have using Choi's multiplicative domain trick that 
$$\Psi(\eta) = \hat{\Psi}(p) \hat{\Psi}(\eta)  \hat{\Psi}(p) = \hat{\Psi}(p \eta p) = 
\hat{T}(p \eta p).$$
Thus $T'(a) = \Phi(a) = \tilde{T}(a)$ for $a \in A$.
 \end{proof}

\subsection{Support projections and HSA's}

If $p$ is a Hilbert space projection on a Hilbert space $H$, and $x$ is any operator on $H$
with $px + xp = 2x$, then $px = xp = x$ by (\ref{qox}).   It follows that 
the `Jordan support' (the smallest projection with  $px + xp = 2x$)  of a real positive operators $x$
on $H$  is the usual support projection of $x$ in $B(H)$ if that
exists (which means that the right and left support projections in $B(H)$ agree).  
This support projection does exist for real positive operators $x$,
as is shown in \cite[Section 3]{BBS}.
Indeed for a Jordan operator algebra $A$ on $H$, if $x \in {\mathfrak r}_A$ and $px = x$  or $x = px$ 
for a projection $p$ on $H$  
then it is an exercise (using the fact that $x + x^* \geq 0$)
that $pxp = x$ ($= px = xp$).  Thus the left and right support projections
on $H$ 
agree, and this will also
be the smallest projection with  $pxp = x$.

If $x$ is an element of a Jordan subalgebra $A$  
we may also consider the  Jordan
support projection in $A^{\ast\ast},$ if it exists,
namely the smallest projection $p \in A^{\ast\ast}$ such that $px + xp = 2x$.    
Recall that if the left and right support projections of $x$ in $A^{\ast\ast}$ (that 
is the smallest projection in $A^{**}$ such that 
$px = x$ or $xp = x$ respectively) coincide, then we call this the {\em  support projection} of $x$, 
and write it as $s(x)$. If this holds, then $s(x)$ clearly also equals the Jordan
support projection in $A^{\ast\ast}$.

The following result is a Jordan operator algebra version of results in  \cite[Section  2]{BRI}.

  \begin{lemma}
\label{op}
For any Jordan operator algebra $A$, if $x\in {\mathfrak r}_A,$ with $x\neq 0$, then the left support projection of $x$ in $A^{**}$ equals the right support projection equals the  Jordan
support projection, and also equals $s({\mathfrak F}(x))$
where ${\mathfrak F}(x) = x(1+x)^{-1} \in \frac{1}{2} {\mathfrak F}_A$.
This also is the weak* limit of the net $(x^{\frac{1}{n}}),$ and is an 
$A$-open projection in $A^{\ast\ast}$, and is open in $B^{**}$ in the sense 
of Akemann {\rm \cite{Ake2}} if $A$ is a Jordan subalgebra of a $C^*$-algebra $B$.
If $A$ is a Jordan subalgebra of $B(H)$ then the left and right support projection of $x$ in $H$ are also equal,
and equal the Jordan support projection there.
\end{lemma}

\begin{proof}
Viewing oa$(x) \subset A$, as in the  operator algebra case  the identity $e$ of oa$(x)^{**}$ is a projection, and  $e = 
{\rm w^*lim} \, x^{\frac{1}{n}} \in \overline{xAx}^{w^*} \subset A^{**}$ with $ex =xe=x$.
Also,  any projection in $B^{**}$ with
$px = x$ or $x p = x$ satisfies $p e = e$.  So $e$ is the support projection $s(x)$ in $A^{**}$ or in $B^{**}$,
and by the discussion above the lemma also equals the Jordan support projection.    It is
$A$-open and open in the sense of Akemann since $x^{\frac{1}{n}} = e x^{\frac{1}{n}} e \to e$ weak*.
To see the support projection equals $s(x(1+x)^{-1}),$ simply note that $px=x$ iff $px(1+x)^{-1}=x(1+x)^{-1}.$ 
That ${\mathfrak F}(x) = x(1+x)^{-1} \in \frac{1}{2} {\mathfrak F}_A$ is as in the argument above 
\cite[Lemma 2.5]{BRord}.  

Suppose that  $\pi: A^{\ast\ast}\to B(H)$ is the natural weak*-continuous Jordan homomorphism extending the inclusion map on $A$.  Then 
 $\pi(p)$ is an orthogonal projection in $B(H)$ with  
$$\pi(p)x+x\pi(p)=\pi(px+xp)=\pi(2x)=2x.$$   Then  
$\pi(p)x=x\pi(p)=x$ by (\ref{qox}), so that $P\leq \pi(p)$ where
$P$ is  the Jordan support projection
of $x$ in $B(H)$.  If $x_t\to p$ weak* with $x_t\in xAx,$ then 
$$P\pi(p)=\lim_t Px_t=\lim_t x_t=\pi(\lim_t x_t)=\pi(p),$$
so $\pi(p)\leq P.$
Hence $\pi(p) = P$.   That the left and right support projection of $x$ in $H$ are also equal to $P$ 
for real positive $x$ is
discussed  above the lemma. 
\end{proof}

\begin{corollary} \label{rop2}  If $A$ is a closed Jordan subalgebra of a C*-algebra $B$, and $x\in {\mathfrak r}_A,$ then the support projection of $x$ computed in $A^{**}$ is the same, via the canonical embedding $A^{**}\cong A^{\perp\perp}\subset B^{**}$, as the support projection of $x$ computed in $B^{**}$. 
\end{corollary}

 If $x \in {\mathfrak F}_A$ for any Jordan operator algebra $A$ then  $x\in {\rm oa}(x)$, the closed
(associative) algebra generated by $x$ in $A$, and $x = \lim_n \, x^{\frac{1}{n}} x$,  so that $\overline{xAx}=\overline{xA^1x}.$   

\begin{lemma}
\label{rop}
For any Jordan operator algebra $A$, if $x\in {\mathfrak r}_A,$ with $x\neq 0$, then $\overline{xAx}$ is a HSA, $\overline{xAx}=s(x) A^{**} s(x) \cap A$ and $s(x)$ is the support projection of $\overline{xAx}.$ 
 If $a = {\mathfrak F}(x) = x (1+x)^{-1} \in \frac{1}{2} {\mathfrak F}_A$ then 
$\overline{xAx} = \overline{aAa}$.  
This HSA has
$(x^{\frac{1}{n}})$ as a partial  cai, and this cai is in ${\mathfrak r}_A$ (resp.\ in ${\mathfrak F}_A$, 
in  $\frac{1}{2} {\mathfrak F}_A$) if $x$ is real positive (resp.\ in ${\mathfrak F}_A$, in  $\frac{1}{2} {\mathfrak F}_A$). 
If also  $y \in {\mathfrak r}_A$ then $\overline{xAx} \subset \overline{yAy}$ iff $s(x) \leq s(y)$.   
\end{lemma}

\begin{proof}  This follows as in the operator algebra case (see the cited papers of the first author and Read), and also uses the fact that if  $x\in {\mathfrak r}_A$
then $x$ has roots in oa$(x)$, so that $x \in \overline{xAx}$.  We have that
$\overline{xAx}^{w*} = s(x) A^{**} s(x)$, so that $\overline{xAx}  = s(x) A^{**} s(x) \cap A$.   

The last assertion follows from the above and the remark after Proposition \ref{hsaii}.    
\end{proof}

\begin{lemma} \label{erphix}  
Let $A$ be an   approximately unital  Jordan operator algebra. If $x\in {\mathfrak F}_A$, then for any state $\varphi$ of $A$, $\varphi(x)=0$ iff $\varphi(s(x))=0.$	
\end{lemma}

\begin{proof}  This follows from 
the matching operator algebra result, since
states on $A$ are precisely the 
restriction of states on $C^*(A)$.  \end{proof}

\begin{lemma}
\label{srp}
Let $A$ be an   approximately unital  Jordan operator algebra. For  $x\in {\mathfrak r}_A$,  consider the 
conditions \begin{itemize} \item [(i)] 
$\overline{xAx}=A$.  \item [(ii)]  $s(x)=1_{A^{**}}.$   \item [(iii)] 
$\varphi(x)\neq 0$ for every state of $A$.
 \item [(iv)]  $\varphi({\rm Re} (x))>0$ for every state $\varphi $ of $C^*(A)$.
\end{itemize}   Then {\rm (iv)} $\Rightarrow$  {\rm (iii)} $\Rightarrow$  {\rm (ii)} $\Leftrightarrow$ {\rm (i)}.
If $x\in {\mathfrak F}_A$ all these conditions are equivalent.
\end{lemma}

\begin{proof}  This as in \cite[Lemma 2.10]{BRI} and the discussion 
of the ${\mathfrak r}_A$ variant of that result above 
\cite[Theorem 3.2]{BRord}; part of it following from Lemma \ref{erphix}.  
\end{proof}

An element in ${\mathfrak r}_A$ with ${\rm Re}(x)$  strictly positive in the $C^*$-algebraic sense 
as in (iv) of the previous result  
will be called {\em strictly real positive}.     Many of the results
on strictly real positive elements from \cite{BRI,BRord} will be 
true in the Jordan case, with the same proof.   For example we will
have the Jordan operator algebra version of \cite[Corollary 3.5]{BRord} 
that if  $x$ is strictly real positive
then so is $x^{\frac{1}{k}}$ for $k \in \Ndb$.

\begin{lemma} \label{supop}  Let $A$ be a Jordan operator algebra, a subalgebra
of a $C^*$-algebra $B$.   
\begin{itemize}  \item [(1)] The support projection of a HSA $D$ in $A$  equals
$\vee_{a \in {\mathfrak F}_D} \, s(a)$ (which equals
$\vee_{a \in {\mathfrak r}_D} \, s(a)$).
\item [(2)]
The supremum  in $B^{**}$ (or equivalently, in the diagonal $\Delta(A^{**})$) of any collection $\{ p_i \}$
of $A$-open projections is $A$-open, and is the support projection of the smallest HSA containing all the HSA's corresponding to the $p_i$.
\end{itemize} 
\end{lemma}

\begin{proof}  
Let $\{ D_i : i \in I \}$ be a collection of HSA's in a 
Jordan operator algebra $A$.   Let $C$ be the convex hull
of $\cup_{i \in I} \, {\mathfrak F}_{D_i}$, which is
a subset of ${\mathfrak F}_{A}$.   Let $D$ be the
closure of $\{ a A a : a \in C \}$.   Since  any $a \in {\mathfrak F}_{D_i}$ has a cube root in 
${\mathfrak F}_{D_i}$, 
${\mathfrak F}_{D_i}$ and $C$ are subsets of $ \{ a A a : a \in C \} \subset D$.
  We show that $D$ is a subspace.  If $a_1, a_2 \in C$ then $a = \frac{1}{2}(a_1 + a_2) \in C$.
We have $s(a_1) \vee s(a_2) = s(a)$, since this is true
with respect to oa$(A)$ (this follows for example from 
\cite[Proposition 2.14]{BRI} and Corollary 2.6), and 
$A^{**}$ is closed under meets 
and joins.   Hence $a_1 A a_1 + a_2 A a_2 \subset 
s(a) A^{**} s(a) \cap A = \overline{aAa}$.   So  $D$ is a subspace. 
Moreover
$D$ is an inner ideal since $aba A aba \subset aAa$ for 
$a \in C, b \in A$.   

For any finite set $F = \{ a_1 , \cdots ,
a_n \} \subset \{ a A a : a \in C \}$, a similar argument shows
that there exists $a \in C$ with 
$F \subset a A a$.   Hence $a^{\frac{1}{n}} a_k \to a_k$
and $a_k a^{\frac{1}{n}} \to a_k$ for all $k$.  
It follows that $D$ is approximately unital, and is a HSA.

Clearly $D$ is the smallest HSA containing all the $D_i$,
since any HSA containing all the $D_i$ would contain $C$
and $\{ a A a : a \in C \}$.
If $p$ is the support projection of $D$ then 
$f = \vee_{a \in C} \, s(a) \leq p$.   Conversely, 
if $a \in C$ then $aAa \subset fA^{**} f$.
Hence $D$ and $D^{\perp \perp} = p A^{**} p$ are contained in 
$fA^{**} f$, so that $p \leq f$ and $p = f$.     Of course $D = p A^{**} p \cap A$.

In particular, when $I$ is singleton we see that the
support projection of a HSA $D$ equals
$\vee_{a \in {\mathfrak F}_D} \, s(a)$.
This proves (1) (using also the fact from Lemma \ref{op}
that $s(a) = s({\mathfrak F}(a))$ for $a \in {\mathfrak r}_D$).

For (2), if $p_i$ is the support projection of
$D_i$ above then $r = \vee_{i \in I} p_i \leq p$ clearly.
On the other hand, if $a \in C$ is a convex combination
of elements of ${\mathfrak F}_{D_{i_j}}$ for $j = 1, \cdots, m$,
then $r a r = a$, so that $s(a) \leq r$.
 This implies by the above that $p \leq r$ and $p = r$.
So suprema in $B^{**}$ of collections of $A$-open projections are $A$-open.
The last assertion is clear from the above.  
\end{proof}

\begin{remark} The intersection of two inner ideals in a Jordan operator algebra $A$ is a inner ideal, and is the largest inner ideal contained in the two (this is not true with `inner ideals' replaced by HSA's, not even in the associative 
 operator algebra case where this would correspond to a false statement about right ideals with left 
approximate identities--see
\cite[Section 5.4]{BZ}).  
\end{remark}

 As in the operator algebra case, 
we may use ${\mathfrak r}_A$ and ${\mathfrak F}_A$ somewhat 
interchangably in most of the next several results.
This is because of facts like: if $a \in {\mathfrak r}_A$ then a Jordan subalgebra
of $A$ contains $a$ iff it contains $x = a (1+a)^{-1}
\in \frac{1}{2}  {\mathfrak F}_A$.    Indeed $x \in {\rm oa}(a)$ and 
since $x + xa = a$ we have $a = x(1-x)^{-1} \in {\rm oa}(x)$  as in the proof of
\cite[Lemma 2.5]{BRord} (the power series for $(1-x)^{-1}$ converges by the Neumann
lemma since $\Vert x \Vert < 1$, as follows from \cite[Lemma 2.5]{BRord} with $A$ 
replaced by ${\rm oa}(a)$).      Also, 
$\overline{xAx} = \overline{aAa}$ by Lemma \ref{rop}.     

\begin{lemma}
\label{smallestH}  For any Jordan operator algebra $A$,  if
$E \subset {\mathfrak r}_A$  then the smallest hereditary  subalgebra of $A$ containing $E$ is
$pA^{**} p \cap A$ where $p = \vee_{x \in E} \, s(x)$.   
\end{lemma} \begin{proof}  
 By Lemma  
\ref{supop}, $pA^{**} p \cap A$ is a
 hereditary subalgebra of $A$, and it  contains $(a_i)$.   Conversely if
$D$ is a hereditary subalgebra of $A$ containing $(a_i)$ then $D^{\perp \perp}$ contains $p$ by the usual
argument, so  $pA^{**} p \subset D^{\perp \perp}$  and
$pA^{**} p \cap A \subset D^{\perp \perp} \cap A = D$.
\end{proof}

As in \cite{BRI} (above Proposition 2.14 there), the correspondence between  a HSA $D$ and its support projection is a bijective 
order embedding from the lattice of HSA's of a Jordan operator algebra $A$ and the lattice of $A$-open projections
in $A^{**}$ (see e.g.\ \cite{N} for a JB-algebra variant of this).    Write $Q(A)$ for the quasistate space of $A$, that is the set of states multiplied by numbers in $[0,1]$.    In the next several results we will be using facts from Section
\ref{aifnl}, namely that states on an approximately unital Jordan subalgebra $A$ are restrictions of states on 
a containing $C^*$-algebra $B$.

\begin{theorem} \label{wsfac}
Suppose that $A$ is an approximately unital  Jordan subalgebra of a $C^*$-algebra $B.$ If $p$ is a nontrivial projection in $A^{\perp\perp}\cong A^{\ast\ast},$ then the following are equivalent:	
\begin{itemize}
\item [(i)] $p$ is open in $B^{\ast\ast}$.
\item [(ii)]	The set $F_p= \{ \varphi \in Q(A): \varphi(p)=0 \}$ is a weak* closed face in $Q(A)$ containing $0$.
\item [(iii)] $p$ is lower semicontinous on $Q(A).$
\end{itemize}
These all hold for $A$-open projections in $A^{**}$, and for such projections
$F_p = Q(A) \cap D^\perp$ where $D$ is the HSA in $A$ supported by $p$. 
\end{theorem}  

\begin{proof}  The first assertions are just as in \cite[Theorem 4.1]{BHN} (and the remark above that result), using the remark above the present 
theorem and fact that $A$-open projections are open with respect to a containing $C^*$-algebra.
For the last assertion,  if $p$ is $A$-open then $Q(A) \cap D^\perp \subset F_p$ since a cai for $D$ converges weak* to $p$.   Conversely if $\varphi \in F_p$ then by the fact above the theorem  $\varphi$ extends to a positive functional on a $C^*$-cover of $A$.
One may assume that this is a state,
 and use the Cauchy-Schwarz inequality to see that $|\varphi(x) | = |\varphi(px)| \leq 
\varphi(p)^{\frac{1}{2}} \varphi(x^*x)^{\frac{1}{2}} = 0$ for $x \in D .$
\end{proof}

If $A$ is unital then there is a similar result and proof using the state space $S(A)$ in place of $Q(A)$.  

\begin{proposition} \label{orderfac} Let $A$ be an approximately unital Jordan operator algebra   The correspondence
$p \mapsto F_p$ is a one-to-one order reversing embedding from the $A$-open projections into the lattice of 
weak* closed  faces of $Q(A)$ containing $0$, thus $p_1 \leq p_2$ iff $F_{p_2} \subset F_{p_1}$
for  $A$-open projections $p_1, p_2$ in $A^{**}$.   Similarly there is a 
one-to-one order reversing embedding $D \mapsto F_p$ from the HSA's  in $A$ 
into the  lattice of  faces of $Q(A)$ above, where $p$ is the support projection of the HSA $D$.  \end{proposition}   \begin{proof}  
Indeed an argument similar to the last part of the last proof shows that 
$F_{p_1} \subset F_{p_2}$ if $p_2 \leq p_1$.   Conversely, if $F_{p_1} \subset F_{p_2}$, then by the argument 
above \cite[Proposition 2.14]{BRI} this implies a similar inclusion but in $Q(B)$, which by the $C^*$-theory gives
$p_2 \leq p_1$.  The last assertion follows from the first and the bijection between HSA's and their support projections.
\end{proof}

\begin{corollary}
\label{ishsa}
Let $A$ be any Jordan operator algebra (not necessarily with an identity or approximate identity.) Suppose that $(x_k)$ is a sequence in ${\mathfrak F}_A,$ and that $\alpha_k\in (0,1]$ and $\sum_{k=1}^{\infty} \, \alpha_k=1$. Then the HSA generated by all the $\overline{x_kAx_k}$ equals $\overline{zAz}$	where $z = \sum_{k=1}^{\infty} \, \alpha_k
x_k \in {\mathfrak F}_A$.   Equivalently (by  Lemma {\rm \ref{supop}}), $\vee_k \, s(x_k)=s(z).$
\end{corollary}
\begin{proof}   This follows similarly to the operator algebra case in \cite[Proposition 2,14]{BRI}.
 If $x \in {\mathfrak F}_A$ then  $\overline{xAx}=\overline{xA^1x}$ as we said above Lemma \ref{rop}.  So we may assume that $A$ is unital.
As in  the operator algebra case $F_{s(z)} = \cap_{k=1}^{\infty} \, F_{s(x_k)}$, which implies by
the lattice isomorphisms above the corollary that $\vee_k s(x_k) \geq s(z),$ and that the smallest 
HSA $D$ containing all the $\overline{x_k A x_k}$ contains $\overline{zAz}$.   Conversely, $z \in \sum_k \overline{x_k A x_k}
\subset D$, so that $s(z) \leq \vee_k s(x_k)$, and so we have equality.
\end{proof}

\begin{theorem}  \label{Ididnt}   \begin{itemize}
\item [(1)]  If $A$ is an associative operator algebra then the HSA's  (resp.\ right ideals with left contractive
approximate identities) in $A$ 
are precisely the sets of form  $\overline{EAE}$  (resp.\ $\overline{EA}$) for some  $E \subset {\mathfrak r}_A$.
The latter set is the smallest HSA (resp.\ right ideal with left 
approximate identity) of $A$ containing $E$.  
\item [(2)]  If $A$ is a Jordan operator algebra then the HSA's in $A$ 
are precisely the sets of form  $\overline{\{ x A x : x \in {\rm conv}(E) \}}$  for some  $E \subset {\mathfrak r}_A$. 
The latter set equals $\overline{\{ x a y + ya x : x, y \in {\rm conv}(E) , a \in A \}}$, and 
 is the smallest HSA of $A$ containing $E$  (c.f.\ Lemma {\rm \ref{smallestH}}).  
 \end{itemize} 
\end{theorem} 

\begin{proof}  We had thought that 
(1) was explicitly in \cite{BRI}, and said  essentially this in e.g.\ \cite{Bsan} after the proof of Corollary 7.10.
Also the result follows from the method of proof we have used several times for different results in our earlier papers, e.g.\ 
\cite[Theorem 7.1]{BOZ}.   Indeed one direction
is obvious by taking $E$ to be a real positive cai for the HSA or right ideal.    For the other direction,
we may assume that $E \subset \frac{1}{2}{\mathfrak F}_A$ by the argument in the first lines of the just mentioned
proof.
Note that $D = \overline{EAE}$  (resp.\ $\overline{EA}$) satisfies $DAD \subset D$ (resp.\ is a right ideal).
For any finite subset $F \subset E$ if $a_F$ is the average of the elements in $F$
then $F \subset \overline{a_F A a_F}$ (resp.\ $F \subset \overline{a_F A}$) since $s(a_F) = \vee_{x \in F} \,
s(x)$ by the operator algebra variant of Corollary
\ref{ishsa}.    
 By a standard argument in our work 
(e.g.\ seen in Lemma \ref{supop}), it is easy to see that $(a_F^{\frac{1}{n}})$ 
will serve as the approximate identity we seek.   Or we can find the latter by the method in the next paragraph.  
The last assertion is fairly obvious.  

(2) \ If $x, y \in{\rm conv}(E)$ and $a, b \in A$ then   $xax + yay \in \overline{z A z}$ where $z = 
\frac{1}{2}(x+y)$ by  Corollary
\ref{ishsa}, and $xax b xax \in D$.   So  $D = \overline{\{ x A x : x \in {\rm conv}(E) \}}$ is a closed inner ideal of $A$.
It follows by the remark before
Proposition \ref{hsaii} that if $x, y  \in {\rm conv}(E)$ then $xax b ycy + ycy b xax \in D$.   Since $x \in \overline{xAx}$ and
$y \in \overline{yAy}$ we have $D = \overline{\{ x a y + ya x : x, y \in {\rm conv}(E) , a \in A \}}$. 
If $F, a_F$ are as in the proof of (1) then $a_F \in {\rm conv}(E)$, and 
$F  \subset \overline{a_F A a_F} \subset D$.     The  
HSA $\overline{a_F A a_F}$ has a J-cai
in $D$, so that there exists $d_{\epsilon,F} \in {\rm Ball}(D)$ such that 
$\| d_{\epsilon,F} x + x d_{\epsilon,F} - 2 x \| < 
\epsilon$ for all $x \in F$.   Hence $D$ has 
$(d_{\epsilon,F})$ as a J-cai.   Again the final assertion is  obvious. 
 \end{proof}

\begin{theorem}
\label{ppn}
Let $A$ be a Jordan operator algebra (not necessarily with an identity or approximate identity.) The HSA's in $A$ are precisely the closures of unions of an increasing net of HSA's of the form $\overline{xAx}$ for $x\in {\mathfrak r}_A$ 
(or equivalently, by an assertion in Lemma {\rm  \ref{rop}} for $x\in {\mathfrak F}_A$).  \end{theorem}

\begin{proof}  Suppose that $D$ is a HSA.
The set of 
HSA's $\overline{a_F A a_F}$ as in the last proof, indexed by the finite subsets $F$ 
of ${\mathfrak F}_D$, is an increasing net.
Lemma \ref{supop} shows that the closure of the union 
of these HSA's is $D$.  
\end{proof}

A HSA $D$ is called ${\mathfrak F}$-{\em principal} if $D=\overline{xAx}$ for some $x\in {\mathfrak F}_A$.  By an assertion
in Lemma \ref{rop} we can also allow $x\in {\mathfrak r}_A$ here.
   Corollary \ref{ishsa} says that
the HSA generated by a countable number of ${\mathfrak F}$-principal HSA's is ${\mathfrak F}$-principal.

\begin{theorem}
\label{spp}
Let $A$ be any Jordan operator algebra (not necessarily with an identity or approximate identity.) 
Every separable HSA or HSA with a 
countable cai is ${\mathfrak F}$-principal.   
\end{theorem}
\begin{proof} If $D$ is a  HSA with a
countable cai, then $D$ has a countable partial cai $(e_n) \subset
\frac{1}{2} {\mathfrak F}_D$.  Also $D$ is generated by the 
HSA's $\overline{e_n A e_n}$, and so 
$D$ is ${\mathfrak F}$-principal by the last result.    For the separable case, note that
any separable approximately unital Jordan operator algebra has a countable cai. 
\end{proof}

\begin{corollary}
\label{sop} If $A$ is a 
separable Jordan operator algebra, then the 
$A$-open projections in $A^{**}$ are  precisely the $s(x)$ for $x\in {\mathfrak r}_A.$	
\end{corollary}

\begin{proof}   If $A$ is separable then so is any HSA.  
So the result follows from  Theorem \ref{spp},  
Lemma \ref{rop}, and Corollary \ref{rop2}.
\end{proof}

\begin{theorem}
\label{cpc}
Let $A$ be any approximately unital Jordan operator algebra. 
 The following are equivalent
\begin{enumerate}
\item [(i)] $A$ has a countable Jordan cai.
\item [(ii)] There exists $x\in {\mathfrak r}_A$ such that $A=\overline{xAx}.$
\item [(iii)] There is an element $x$ in ${\mathfrak r}_A$ with $s(x)=
1_{A^{**}}.$
 \item [(iv)] $A$ has a strictly real positive element in ${\mathfrak r}_A$.
\end{enumerate}	If $A$ is separable then these all hold.
\end{theorem}

\begin{proof}   The equivalence  of (ii), (iii), and (iv) comes from Lemma \ref{srp}
and the reasoning for 
\cite[Theorem 3.2]{BRord}.   These imply (i) since 
(a scaling of) $(x^{\frac{1}{k}})$ is a countable partial cai.
The rest follows from Theorem \ref{spp} applied to $A = D$.
\end{proof}

\begin{remark} We remark again that one may replace
${\mathfrak r}_A$ by ${\mathfrak F}_A$ in the last several results.
\end{remark}

 \begin{theorem}
An approximately unital Jordan operator algebra with no countable Jordan cai, has nontrivial HSA's.
\end{theorem}
\begin{proof}
If $A$ has no countable  cai then by Theorem  \ref{cpc} 
for any nonzero $x\in {\mathfrak F}_A$, we have $A \neq \overline{xAx}$.
The latter is a nontrivial HSA in $A$.
\end{proof}

\subsection{M-ideals}   

\begin{theorem} \label{mid}  
Let $A$ be an approximately unital Jordan operator algebra.
\begin{itemize}
\item [(1)] The $M$-ideals in $A$ are the complete 
$M$-ideals. These are exactly the 
closed Jordan ideals in $A$ which are approximately unital.  
\item [(2)] The $M$-summands in $A$ are the complete $M$-summands. These are exactly the sets  $Ae$ for a 
projection $e$  in $JM(A)$   (or equivalently in $M(A)$) such that $e$ commutes with all
elements in $A$).    If $A$ is unital then these are the closed Jordan ideals in $A$ which possess a Jordan
 identity of norm $1$.  
\item [(3)] The right $M$-ideals in $A$ are of the 
form $J = p A^{\ast\ast} \cap A$, 
where $p$ is a projection in $M(A^{**})$ with $J^{\perp \perp} = p A^{\ast\ast}$.
Each right $M$-ideal in $A$ is a Jordan subalgebra with a left $C^*_e(A)$-partial  cai.
\item [(4)] The right $M$-summands in $A$ are exactly the sets $pA$ for an idempotent contraction $p\in M(A).$  
 \end{itemize}
\end{theorem}
\begin{proof}
(4) \ By 4.5.15 in \cite{BLM}, the left $M$-projections are the 
projections in the left multiplier algebra ${\mathcal M}_\ell(A)$ of
 \cite{BLM}. Hence, the right $M$-summands in $A$ are exactly the sets $pA$ for an idempotent contraction $p\in M_l(A) = LM(A).$ 
So $p$ may be regarded as a projection in $LM(A)$.   However by 
Lemma \ref{pin} any projection in $LM(A)$ is in $M(A)$.

(3) \ If $J$ is a right $M$-ideal then $J^{**}=J^{\perp\perp}=\overline{J}^{w^*}$ is a right $M$-summand. Hence by (4), $J^{\perp\perp}=pA^{\ast\ast}$,
where $p$ is a projection in $M(A^{**})$.
Thus $J=J^{\perp\perp}\cap A=pA^{\ast\ast}\cap A.$ 
It follows that $J$ is a  Jordan subalgebra.
Note that if $(e_t) \subset J$ with $e_t \to p$ weak*, then 
$e_t x \to p x = x$ (products in $C^*_e(A)^{**}$)
for all $x \in J$.   Thus as in Lemma \ref{jcai} a convex combination 
of the $e_t$ are a  left partial cai for $J$.   

(2) \ If $e$ is a projection in $M(A)$ commuting with $A$, then since
$ea = eae \in A$ we 
see that left multiplication by $e$ is in the algebra $M_l(A)$ mentioned above, 
and $eA$ is a right $M$-summand by (4).
Similarly $eA = Ae$ is a left
 $M$-summand by the left-handed version of (4).
  So $eA$ is a complete $M$-summand by 
\cite[Proposition 4.8.4 (2)]{BLM}. 

Conversely, suppose that $P$ is an $M$-projection on $A.$ First suppose that $A$ is unital.
Set $z=P(1)$ and follow the proof of \cite[Theorem 4.8.5 (2)]{BLM},
to see that $z$ is Hermitian in $A$ and $z^2=z$, so that
 $z$ is a projection in $A.$  That argument goes on to 
show that if $\varphi$ is any state with $P^*(\varphi)\neq 0,$ and if $\psi=\frac{P^*(\varphi)}{\Vert P^*(\varphi)\Vert}$, then $\psi$ is a state on $A.$ 
As we said earlier, we can extend $\psi$ to a state $\tilde{\psi}$ on some 
$C^*$-algebra generated by $A.$  As in the argument we are following
 we obtain, for any $a\in A,$ that 
$\vert\tilde{\psi}(a(1-z))\vert^2\leq \tilde{\psi}(aa^*)\psi(1-z)=0$, 
so that
 $\tilde{\psi}(a(1-z))=0.$ Similarly, $\tilde{\psi}((1-z)a)=0.$ Hence, $\varphi(P(a(1-z)+(1-z)a))=0.$ Since this holds for any state, we have
 $P(a(1-z)+(1-z)a)=0.$ Therefore, $(1-z) \circ A\subset (I-P)(A)$.
By symmetry we have $z\circ A\subset P(A).$ If $a\in A,$ then $$a=\frac{az+za}{2}+\frac{a(1-z)+(1-z)a}{2},$$ so that $P(a)=\frac{az+za}{2} = z \circ a$.   That $P$ is idempotent yields the formula
$zaz = z \circ a$.   So $z$ is central in $A$, and  $az=za=zaz,$ and $P(A) = zA$.

Next, if $A$ is not unital consider the $M$-projection $P^{\ast\ast}$ on $A^{\ast\ast}$.  By the unital case $P^{\ast\ast}(\eta) = z \eta = \eta z  = z \eta z$ for all $\eta \in A^{**}$, for a central projection $z \in A^{**}$.   We have
$z a + az = 2P(a) \in A$ for $a \in A$, so that $z \in M(A)$, and $P(A) = zA$.  

Finally suppose that $J$ is a closed Jordan ideal in $A$ which possesses a Jordan identity $e$ 
of norm $1$.   Then  $ex = x = xe$ for all $x \in J$, as in the proof of Lemma \ref{jcai}.
Also $eAe \subset J = eJe \subset eAe$.   So $J = eAe$.  
Also $J = e \circ A$, so $e a + ae   = eae +ae$, and so $ea = eae$.  Similarly $ae = eae = ea$, so 
$e$ is central in $A$.   The rest is clear.
 
(1) \ If $J$ is an approximately unital closed Jordan ideal in $A,$ 
then $J^{\perp \perp}$ is by the usual 
approximation argument a unital weak* closed Jordan ideal in $A^{**}$.
So by  (2) we have $J^{\perp \perp}$ is the $M$-summand $p A^{**}$ for a central projection $p\in A^{\ast\ast}$.
So $J$ is an $M$-ideal. Conversely, if $J$ is  an $M$-ideal, then $J^{\perp\perp}$ is an $M$-summand in $A^{\ast\ast}.$ By (2), there exist a central projection $e\in A^{\ast\ast}$ such that $J^{\perp\perp}=eA^{\ast\ast}$  and $e 
\in M(A^{**})$.
Note that $e\in J^{\perp\perp}.$   By a routine argument similar to the associative case,
$J$ is a Jordan ideal with partial cai.
\end{proof}

\begin{corollary} \label{nona}
	A subspace $D$  in a Jordan operator algebra $A$ is an approximately unital closed Jordan ideal in $A$ iff there exists some open central projection $p$ in $A^{\ast\ast}$,  
such that $D=pA^{\ast\ast}p \cap A$. 
\end{corollary}

\begin{proof}   If  $D$ is an approximately unital  closed  Jordan ideal in
$A$ then it is an approximately unital  closed  Jordan ideal in
$A^1$.
The proof of  Theorem \ref{mid}  (1) 
shows that $D^{\perp\perp} 
= p A^{\ast\ast}$, for a projection $p$ in $D^{\perp\perp}
\subset A^{\ast\ast}$ 
(the weak* limit of a cai for $D$).  Also $p$ is central in
$(A^1)^{**}$, hence in $A^{\ast\ast}$.  
Clearly $D=pA^{\ast\ast}p \cap A$, so $p$ 
is open, and  $D$ is a HSA.

Conversely, if $p$ is an open central projection in $A^{\ast\ast}$,
then $p$ is  an open central projection in $(A^1)^{\ast\ast}$.
Since $p \eta = \eta p = p \eta p \in A^{\ast\ast}$ for $\eta \in A^{\ast\ast}$, 
we have 
$D =pA^{\ast\ast}p \cap A$
is a HSA  and in particular is approximately unital.   It is easy to see that $D$ is a closed Jordan ideal since $p$ is central.  
\end{proof}  

 \begin{proposition}  \label{isJq}
If $J$ is an approximately unital closed two-sided Jordan ideal in  a
 Jordan operator algebra $A,$ 
then $A/J$ is (completely isometrically isomorphic to)
a Jordan operator algebra.  \end{proposition}  

  \begin{proof}  Since $A/J \subset A^1/J$ we may assume that
$A$ is unital.
By graduate functional analysis 
$$A/J \subset A^{**}/ J^{\perp \perp} =  A^{**}/ eA^{**} = e^\perp A^{**} ,$$
where $e$ is the central support projection of $J$.  The ensuing embedding 
$A/J \subset e^\perp A^{**}$ is the map $\theta(a + J ) = e^\perp \, a$.   Note that 
$\frac{1}{2}(ab + ba) + A$ maps to $\frac{1}{2} e^\perp (ab + ba) = \theta(a) \circ \theta(b)$.
So $\theta$ is a completely isometric Jordan homomorphism into the 
Jordan operator algebra $A^{**}$, so $A/J$ is completely isometrically isomorphic to
a Jordan operator algebra.
\end{proof}

Clearly any  approximately unital Jordan operator algebra $A$ is an $M$-ideal in its unitization, or in $JM(A)$. As in \cite[Proposition 6.1]{BRI} we have:
  
\begin{proposition} \label{qsof}
If $J$ is a closed Jordan ideal in a Jordan operator algebra $A,$ and if $J$ is approximately unital, then $q({\mathfrak F}_A)={\mathfrak F}_{A/J},$ where $q:A\to A/J$ is the quotient map.
\end{proposition}
\begin{proof}  By Propositions  \ref{isJq} and \ref{MeyerJ1}   
we can extend $q$ to a contractive unital Jordan homomorphism
from $A^1$ to a unitization of $A/J$, and then it is
easy to see that $q({\mathfrak F}_A) \subset {\mathfrak F}_{A/J}$.

For the reverse inclusion note that $J$ is an $M$-ideal in $A^1$ by Theorem
\ref{mid} (1).  We may then proceed as in the proof of
\cite[Proposition 6.1]{BRI}.  \end{proof}
 
 Under the conditions of the last proposition we can also show that $q({\mathfrak r}_A)={\mathfrak r}_{A/J}$
by using the method of \cite[Corollary 8.10]{BOZ}.    This will be presented in 
\cite{ZWdraft}. In the next section we will discuss 
the role of  real positivity and ${\mathfrak r}_A$ further.

\section[More on real positivity]{More on real positivity  in Jordan operator algebras}

The ${\mathfrak r}$-{\em ordering} is simply the order $\preccurlyeq$ induced by the above closed cone; that is
$b \preccurlyeq a$ iff $a - b \in {\mathfrak r}_A$.  If $A$ is a Jordan subalgebra of a Jordan operator algebra $B$, we mentioned earlier that ${\mathfrak r}_A \subset {\mathfrak r}_B$.
If $A, B$ are
approximately unital  Jordan subalgebras of $B(H)$ then it follows   from the fact that $A = {\mathfrak r}_A - {\mathfrak r}_A$  (see Theorem \ref{havin}) and similarly for $B$,  
that $A \subset B$ iff ${\mathfrak r}_A \subset {\mathfrak r}_B$.   As in \cite[Section 8]{BRI},
${\mathfrak r}_A$ contains no idempotents which are not orthogonal projections,
and no nonunitary isometries.  
In \cite{BRII} it is shown that $\overline{{\mathfrak c}_A} = {\mathfrak r}_A$.
Also ${\mathfrak r}_A$ contains no nonzero elements with square zero.  Indeed if $(a+ib)^2 = a^2 -b^2
+ i (ab+ba) = 0$ with $a \geq 0$ and $b = b^*$ then $a^2 = b^2$ so that $a$ and $b$ commute.
Hence $ab = 0$ and $a^4  = a^2 b^2 = 0$.   So $a = b = 0$.

\begin{theorem} \label{havin}  Let  $A$ be a Jordan  operator algebra which generates a $C^*$-algebra
$B$, and let ${\mathcal U}_A$ denote the open unit ball $\{ a \in A : \Vert a \Vert < 1 \}$.  The following are equivalent:
\begin{itemize} \item [(1)]   $A$ is approximately unital.
 \item [(2)]  For any positive $b \in {\mathcal U}_B$ there exists  $a \in {\mathfrak r}_A$
with $b \preccurlyeq a$.
 \item [(2')]  Same as {\rm (2)}, but also $a \in \frac{1}{2}  {\mathfrak F}_A$ and nearly positive.
\item [(3)]   For any pair 
$x, y \in {\mathcal U}_A$ there exist   nearly positive
$a \in \frac{1}{2}  {\mathfrak F}_A$
with $x \preccurlyeq a$ and $y \preccurlyeq a$.
\item [(4)]    For any $b \in {\mathcal U}_A$  there exist   nearly positive
$a \in \frac{1}{2}  {\mathfrak F}_A$
with $-a \preccurlyeq b \preccurlyeq a$. 
\item [(5)] For any $b \in {\mathcal U}_A$  there exist 
$x, y \in  \frac{1}{2}  {\mathfrak F}_A$ 
with $b = x-y$.   
\item [(6)]  ${\mathfrak r}_A$ is a generating cone (that is, $A = {\mathfrak r}_A - {\mathfrak r}_A$). 
\item [(7)]  $A = {\mathfrak c}_A - {\mathfrak c}_A$. 
\end{itemize}  \end{theorem} 

\begin{proof}   This is identical to the proof of \cite[Theorem 2.1]{BRord} except in the proof that  (6) implies (1), and (1) 
implies (5).   For the former, we use the
fact mentioned above that $A^{**}$ is closed under meets and joins of projections. Then note that the product
$p x$ in that part of  the proof in \cite{BRord} initially is in $B^{**}$, but then we see $px = x$ and can appeal
to Lemma \ref{jcai}.      For    (1) 
implies (5) one can proceed as in \cite[Theorem 6.1]{BOZ},
but appealing to Theorem \ref{kapl} (we remark that there is a typo in the proof of \cite[Theorem 6.1]{BOZ},
the reference cited there should be replaced by e.g.\ \cite[Theorem 5.2]{BOZ}).
\end{proof}  

For the next results let $A_H$ be the closure of the set $\{ aAa : a \in {\mathfrak F}_A \}$. 
We shall show that $A_H$ is the largest approximately unital Jordan subalgebra of $A$.

\begin{corollary} \label{Ahasc2}  For any Jordan operator algebra $A$,
the largest approximately unital Jordan subalgebra of $A$ is
$${\mathfrak r}_A - {\mathfrak r}_A = {\mathfrak c}_A - {\mathfrak c}_A.$$
In particular these spaces are  closed and form a HSA of $A$.

If $A$ is a  weak* closed Jordan operator
algebra then this largest approximately unital Jordan subalgebra is $q A q$ where $q$ is the largest projection in $A$.
This is weak* closed.
 \end{corollary}

\begin{proof}      The proof of Lemma \ref{supop} (see also
 Theorem  \ref{Ididnt} (2) with $E = {\mathfrak F}_A$)
yields $A_H$ is the HSA $pA^{**} p \cap A$ where $p = \vee_{a \in {\mathfrak F}_A} \, s(a)$
is $A$-open.     
Similarly, $A_H$ is the closure of the set $\{ aAa : a \in 
{\mathfrak r}_A \}$.  As in the proof of \cite[Theorem 4.2 and Corollary 4.3]{BRII} we have that 
$A_H$ is the largest approximately unital Jordan subalgebra of $A$ and ${\mathfrak F}_A
= {\mathfrak F}_{A_H}$ and ${\mathfrak r}_A
= {\mathfrak r}_{A_H}$.   By Theorem \ref{havin} we have 
$A_H = {\mathfrak r}_{A_H} - {\mathfrak r}_{A_H} = {\mathfrak r}_A - {\mathfrak r}_A$,
and similarly $A_H = {\mathfrak c}_{A_H} - {\mathfrak c}_{A_H} = {\mathfrak c}_A - {\mathfrak c}_A$.

The final assertion follows just as in \cite[Corollary 2.2]{BRord}. 
 \end{proof}  

As in \cite[Lemma 2.3]{BRord}, and with the same proof we have:

\begin{lemma} \label{mnah}  Let $A$ be any Jordan operator algebra.  Then   for every $n \in \Ndb$, 
$$M_n(A_H) = M_n(A)_H \; , \; \; \; \; \;  {\mathfrak r}_{M_n(A)} = {\mathfrak r}_{M_n(A_H)}
 \; , \; \; \; \; \;  {\mathfrak F}_{M_n(A)} = {\mathfrak F}_{M_n(A_H)}.$$  
\end{lemma}  

 If $S \subset {\mathfrak r}_A$, for a Jordan  operator algebra
$A$,  and if $xy = yx$
for all $x, y \in S$, write joa$(S)$ for the smallest closed Jordan subalgebra of $A$ containing $S$.

\begin{proposition} \label{commt}  If $S$ is any subset of ${\mathfrak r}_A$ for a Jordan  operator algebra
$A$,  
then ${\rm joa}(S)$ is approximately unital.
\end{proposition}

\begin{proof}   Let $C = {\rm joa}(S)$.  Then  ${\mathfrak r}_C = C \cap {\mathfrak r}_A$.
If $x \in S$ then
$x \in {\mathfrak r}_C = {\mathfrak r}_{C_H} 
\subset C_H.$  So $C \subset  C_H \subset C,$    since 
$C_H$ is a Jordan  operator algebra containing $S$.
Hence $C = C_H$, which is
approximately unital.   
\end{proof} 

\begin{lemma} \label{Font}   For any operator algebra $A$, the  ${\mathfrak F}$-transform 
${\mathfrak F}(x) = 1 - (x+1)^{-1} = x (x+1)^{-1}$  maps ${\mathfrak r}_{A}$ bijectively onto the set of elements of $\frac{1}{2}{\mathfrak F}_{A}$
of norm $< 1$.   Thus ${\mathfrak F}({\mathfrak r}_A) = {\mathcal U}_A \cap \frac{1}{2}{\mathfrak F}_{A}$.  \end{lemma} 

\begin{proof}  This follows from part of the discussion above Lemma
\ref{smallestH}. 
 \end{proof} 

We recall that the positive part of the  open unit ball of a $C^*$-algebra
is a directed set, and indeed is a  net which is a positive cai for $B$ (see e.g.\ \cite{P}).  As in 
 \cite[Proposition 2.6 and Corollary 2.7]{BRord}, and with the same proofs we have:

\begin{proposition} \label{isnet}  If $A$ is an approximately unital  Jordan operator algebra, then
${\mathcal U}_A \cap \frac{1}{2} {\mathfrak F}_A$ is a directed set in the $\preccurlyeq$ ordering,
and with this ordering ${\mathcal U}_A \cap \frac{1}{2} {\mathfrak F}_A$  is an increasing 
partial cai for $A$.  \end{proposition}

\begin{corollary}  \label{proz} Let $A$ be an
approximately unital  Jordan operator algebra, and $B$ a $C^*$-algebra generated by $A$.
If $b \in B_+$ with $\Vert b \Vert < 1$ then 
there is an increasing partial cai for $A$ in $\frac{1}{2} {\mathfrak F}_A$, every term of which dominates $b$
(where `increasing' and `dominates'
are in the $\preccurlyeq$ ordering).
\end{corollary} 

\begin{remark}
	Any  Jordan operator algebra $A$ with a countable cai, and in particular
any separable approximately unital Jordan operator algebra $A$, has a commuting  partial cai which is increasing
(for the $\preccurlyeq$ ordering), and 
also in $\frac{1}{2} {\mathfrak F}_A$
and nearly positive.   Namely, by Theorem \ref{cpc}  
  we have $A = \overline{xAx}$ for some $x \in  \frac{1}{2} {\mathfrak F}_A$, and $(x^{\frac{1}{n}})$ is a commuting  partial cai which is increasing by \cite[Proposition
4.7]{BBS}.   
\end{remark}

A ($\Cdb$-)linear map $T : A \to B$ between Jordan operator algebras is {\em real positive} if
 $T({\mathfrak r}_A)
\subset {\mathfrak r}_B$.
We say that $T$ is {\em real completely positive} or  RCP  if the $n$th matrix amplifications $T_n$ are each  real positive.
  It is clear from properties of ${\mathfrak r}_A$ mentioned earlier, that restrictions of real positive
(resp.\ RCP) maps to Jordan subalgebras  (or to unital operator subspaces) are again  real positive
(resp.\ RCP).

\begin{corollary}   \label{iscb}  Let $T : A \to B$ be a linear  map between 
approximately unital  Jordan  operator algebras,  and suppose that  $T$ is 
real positive (resp.\ RCP).   Then $T$ is bounded (resp.\ completely bounded).  
Moreover $T$ extends to a well defined positive map $\tilde{T} : 
A + A^* \to B + B^*  : a + b^* \mapsto T(a) + T(b)^*$. 
\end{corollary} 

\begin{proof}  This is as in \cite[Corollary 2.9]{BRord}.   Note that $T^{**}$ is real positive (using
Theorem \ref{kapl}), and hence by the proof of  \cite[Theorem 2.5]{BBS}   it extends to a positive 
map on an operator system.   Indeed it is completely positive, hence completely bounded, in the matrix
normed case.   Then restrict to  $A + A^*$. \end{proof}

\begin{remark}  \label{aiscb}   Similar results hold on unital operator spaces.  With the same proof idea any real positive
 linear  map on a unital operator space $A$ extends to a well defined positive map on $A + A^*$. 
 It is easy to see that a unital contractive 
linear map on a unital operator space is real positive (this follows 
e.g.\ from the fact that ${\mathfrak r}_A = \overline{\Rdb_+ (1 + {\rm Ball}(A))}$ in
this case). 
As in the well known  lemma  of Smith, if $A$ is an operator space or operator system or approximately unital Jordan  operator algebra, then $u :A\to M_n$ is real $n$-positive if and only if it is RCP.    Indeed if  $u$ is real $n$-positive then by the fact at the start of
this remark (or Corollary   \ref{iscb}) applied to $u_n$, $u$ can be extended to an $n$-positive map  from $A+A^*$ 
(or the sum of the biduals) into $M_n$, and this is completely positive by  the selfadjoint theory \cite[Theorem 6.1]{Pau}. Thus $u$ is RCP.  
In particular,  real positive maps into the scalars are RCP, and as usual it follows from this that 
real positive linear maps into a commutative C*-algebra $C(K)$ are 
RCP.    Indeed if $[a_{ij} + a_{ji}^*] \geq 0$, then 
$[T(a_{ij})(x) + \overline{T(a_{ji})(x)}] \geq 0$ for all $x \in K$.
Thus $[T(a_{ij}) + T(a_{ij})^*] \geq 0$ as desired.
Some of the last mentioned results were found by Da Zhang, a few of these together with the first author.
\end{remark}
\begin{theorem}[Extension and Stinespring Dilation for RCP Maps] 
\label{db9} If $T:A\to B(H)$ is a linear map on an approximately unital Jordan operator algebra,  and if $B$ is a $C^*$-algebra containing $A$, then $T$ is RCP iff $T$ has a completely positive extension $\tilde{T}:B\to B(H)$. This is equivalent to being able to write $T$ as the restriction to $A$ of  $V^*\pi(\cdot)V$ for a $*$-representation $\pi: B\to B(K)$, and an operator  $V: H\to K$. Moreover, this can be done with $\norm{T}=\norm{T}_{cb}=\norm{V}^2$, and this equals $\norm{T(1)}$ if $A$ is unital. 
\end{theorem}

\begin{proof}   The structure of this proof follows 
the analogous results in \cite{BBS,BRI}.   We merely begin the proof: as in the proof of Corollary  \ref{iscb}, $T$ is completely bounded
and $T^{**}$ extends to a completely positive map  $A^{**} + (A^{**})^* \to B(H)$.    
\end{proof}

An $\Rdb$-linear $\varphi : A \to \Rdb$  is said to be 
{\em real positive} if $\varphi({\mathfrak r}_A) \subset [0,\infty)$.   By the usual trick, for any $\Rdb$-linear $\varphi : A \to \Rdb$, there
is a unique $\Cdb$-linear $\tilde{\varphi} : A \to \Cdb$  with Re $\, \tilde{\varphi} = \varphi$,
and clearly $\varphi$ is real positive (resp.\ bounded) iff $\tilde{\varphi}$ is real positive (resp.\ bounded). 

As in \cite[Corollary 2.8]{BRord}, and with the same proof we have:

\begin{corollary}  \label{pr} Let $A$ be an 
approximately unital  Jordan operator algebra, and $B$ a $C^*$-algebra generated by $A$.
Then every 
real positive $\varphi : A \to \Rdb$ extends to a real positive real functional on $B$.  Also, $\varphi$ is bounded.  
\end{corollary} 

 We will write ${\mathfrak c}^{\Rdb}_{A^*}$ for the real dual cone of ${\mathfrak r}_A$, the 
set of continuous $\Rdb$-linear $\varphi : A \to \Rdb$ such that $\varphi({\mathfrak r}_A) \subset [0,\infty)$.
Since $\overline{{\mathfrak c}_A} = {\mathfrak r}_A$, we see that ${\mathfrak c}^{\Rdb}_{A^*}$ 
 is also the real dual cone of ${\mathfrak c}_A$.    It is a   proper cone;
 for if  $\rho, -\rho \in {\mathfrak c}^{\Rdb}_{A^*}$ then
$\rho(a) = 0$ for all $a \in {\mathfrak r}_A$.  Hence $\rho = 0$
by the fact above that the norm closure of ${\mathfrak r}_A - {\mathfrak r}_A$
is $A$.

\begin{lemma}  \label{dualc}    Suppose that $A$ is an approximately unital Jordan operator algebra.
 The real dual cone ${\mathfrak c}^{\Rdb}_{A^*}$  equals 
$\{ t \, {\rm Re}(\psi) : \psi \in S(A) , \, t \in [0,\infty) \}$. 
It also equals the set of restrictions to $A$ of the real parts of
positive functionals on any $C^*$-algebra containing 
(a copy of) $A$ as a closed Jordan subalgebra.   The prepolar of ${\mathfrak c}^{\Rdb}_{A^*}$, which equals its
real predual cone,  is  ${\mathfrak r}_{A}$;  
and  the polar of ${\mathfrak c}^{\Rdb}_{A^*}$,  which equals its
real dual cone,  is  ${\mathfrak r}_{A^{**}}$.   Thus the second dual cone of 
${\mathfrak r}_A$ is ${\mathfrak r}_{A^{**}}$, and hence 
 ${\mathfrak r}_A$ is weak* dense in ${\mathfrak r}_{A^{**}}$. \end{lemma}

One may also develop {\em real states} on Jordan operator algebras analogously to \cite[Section 2.3]{BRord}.   See e.g.\ \cite{ZWdraft}.

Following on Kadison's Jordan algebraic Banach--Stone theorem for
$C^*$-algebras \cite{Kad1}, many authors have proved variants for objects
of `Jordan type'.  
 The following  variant on the main result of Arazy and Solel \cite{AS}
is a `Banach--Stone type theorem for Jordan operator algebras'.

\begin{proposition} \label{bsoa}    Suppose that $T : A \to B$ is an isometric surjection between approximately unital 
Jordan operator algebras.  Then $T$ is  real 
positive if and only if $T$ is a Jordan  algebra homomorphism.
If these hold and in addition $T$ is completely isometric 
then $T$ is real completely positive.  
 \end{proposition}  

\begin{proof}  If $T : A \to B$ is an isometric surjective
Jordan algebra homomorphism between  
unital Jordan operator algebras, 
then $T$ is unital hence real positive by a fact 
in  Remark  \ref{aiscb}.  If $A$ and $B$ are 
nonunital (possibly non-approximately unital) 
then $T$ extends to a unital isometric surjective
Jordan algebra homomorphism between the unitizations, hence is real positive again
by the unital case.  

Conversely, suppose that  $T$ is real positive.  Again 
we may assume that $A$ and $B$ are  unital, 
since by the usual arguments $T^{**}$ is
a real positive  isometric surjection between unital 
Jordan operator algebras.  Then the result follows from 
\cite[Proposition 6.6]{BNpac}.
If $T$ is a completely isometric surjective
Jordan algebra homomorphism then by Proposition \ref{ununau},
$T$ extends to a unital completely isometric surjection between the unitizations, which then
extends by Arveson's lemma
e.g.\ \cite[Lemma 1.3.6]{BLM}
to a unital completely contractive, hence completely positive, map on $A+A^*$.
So  $T$ is real completely positive.
\end{proof}

We close with a 
final 
Banach--Stone type theorem.

\begin{proposition} \label{bsjoa}    Suppose that $T : A \to B$ is a  completely isometric surjection between approximately unital
Jordan operator algebras.  Then there exists a 
completely isometric surjective Jordan  algebra homomorphism
$\pi : A \to B$, and a unitary $u$ with $u, u^* \in M(B)$ with $T = u \pi(\cdot)$.   
 \end{proposition}

  \begin{proof}  
If $A, B$ are unital then this follows from the Banach--Stone type theorem
\cite[Corollary 2.8]{AS}, and indeed this works with the word `completely' dropped. 

In the general case,  
we may extend $T$  to a completely isometric surjection $\rho : 
C^*_e(A) \to C^*_e(B)$ between $C^*$-envelopes. 
This follows by the universal property of the ternary envelope of any 
the fact that this ternary envelope is the $C^*$-envelope (see e.g.\ Corollary \ref{injov}).
 By a well known Banach--Stone  theorem for $C^*$-algebras
(see e.g.\ \cite[Theorem 6.1]{B}), $\rho = v \theta$ for a $*$-isomorphism
$\theta : C^*_e(A) \to C^*_e(B)$ and a unitary  $v \in M(C^*_e(B))$.   
So $T(a) = v \theta(a)$ for $a \in A$.  Also  
by the unital case in the last paragraph 
$T^{**} = u \pi(\cdot)$ for a 
unitary $u \in B^{**}$ with $u^* \in B^{**}$,
and a unital completely isometric 
surjective Jordan homomorphism
$\pi : A^{**} \to B^{**}$.   We have $T(a) = u \pi(a) = v \theta(a)$
 for $a \in A$.  Also we have that $\pi$ is weak* continuous.  
 Represent $C^*_e(B)$ on a Hilbert space nondegenerately in such 
a way that $C^*_e(B)^{**} \subset B(K)$ as a von Neumann algebra.  Then 
we may regard $v \in M(C^*_e(B)) \subset B(K)$ too, and  if $(e_t)$ is a partial cai
for $A$ then $\pi(e_t) \to I_K$ weak* (since $1_{B^{**}} = 1_{C^*_e(B)^{**}} = I_K$).
For $\xi \in K$ we have $T(e_t) \xi = u \pi(e_t) \xi \to u \xi$, and
$T(e_t) \xi = v \theta(e_t) \xi \to v \xi$  
(in WOT).
So $u = v$ and $\pi(a) = \theta(a)$ in $B(K)$, for $a \in A$.   Note that $u \pi(A) = T(A) = B$, so 
$$\pi(A) = u^* B \subset
B^{**} \cap C^*_e(B) = B.$$
Similarly $uB    \subset
B^{**} \cap C^*_e(B) = B,$ so $B \subset u^* B$, and we have $uB = B = \pi(B)$.   
By Lemma \ref{pin}, $u$ and $u^*$ are in $M(B)$.  
\end{proof}  

With a little more work one should be able to remove 
the word `completely' in the last proposition.  

We end 
by mentioning a few of what seem to us to be the most important open questions related to the approach of our paper.

  \begin{enumerate}
\item  [(1)]  Is there a purely internal (not mentioning containing algebras or injective envelopes)
characterization of Jordan operator algebras?
  
\item   [(2)] Is the quotient of a Jordan operator algebra by a general closed ideal 
a Jordan operator algebra?

\item   [(3)] Is the unitization of a general Jordan operator algebra unique
 up to completely isometric isomorphism?  We showed in Corollary \ref{MeyerJ} that it is 
 unique up to isometric isomorphism; see also Proposition \ref{ununau} for the 
approximately unital case.   
 
\end{enumerate}

{\em Acknowledgements.}
  We thank the referees for their good suggestions, and  thank Matt Neal for several conversations and insights that were helpful.


\end{document}